\documentclass[11pt,hyp,]{nyjm}
\usepackage{hyperref}
\hypersetup{nesting=true,debug=true,naturalnames=true}
\usepackage{graphicx,amssymb,upref}

\let\<\langle
\let\>\rangle
\newcommand{\doi}[1]{doi:\,\href{http://dx.doi.org/#1}{#1}}

\let\uml\"

\title[Dancing polygons,  rolling balls, and the Cartan-Engel\ldots]{Dancing polygons,  rolling balls and the Cartan-Engel distribution} 

\author{Gil Bor}  
\address{CIMAT, A.P. 402, Guanajuato, Gto. 36000, Mexico} 
\email{gil@cimat.mx}  

\author{Luis Hern\'andez Lamoneda}
\address{CIMAT, A.P. 402, Guanajuato, Gto. 36000, Mexico} 
\email{lamoneda@cimat.mx}

\thanks{We thank Robert Bryant for informative correspondence and to  Travis Wilse for reading an initial draft and making useful suggestions. We acknowledge support from CONACYT Grant A1-
S-45886.} 



\usepackage{amsthm,amsmath,amsfonts,
array,booktabs, cancel, capt-of, cite, color,graphics,
graphbox, 
float,
import, mathrsfs, mathtools,pdfpages,tikz-cd,  
transparent,url,varwidth, wrapfig}
 
\usepackage[shortlabels ]{enumitem}

\newcommand{\s}{\scriptsize}
\newcommand{\fs}{\footnotesize}
\newcommand{\ttf}{$(2,3,5)$}
\newcommand{\CED}{Cartan-Engel distribution}

\newcommand{\dd}{\vspace{2px}\s\tt}

\newcommand{\Qd}{Q^{\small\tt dan}}
\newcommand{\Dd}{\D^{\small\tt dan}}

\newcommand{\roll}{{\small\tt roll}}
\newcommand{\Qr}{Q^\roll}
\newcommand{\tQ}{{\widetilde{Q}}}
\newcommand{\tQr}{{\tQ}^\roll}

\newcommand{\tDr}{\tD^{\small\tt roll}}
\newcommand{\Dr}{\D^{\small\tt roll}}
\newcommand{\Qo}{Q^{\small\tt oct}}
\newcommand{\Do}{\D^{\small\tt oct}}

\newcommand{\tD}{\widetilde{\D}}

\newcommand{\tQo}{{\tQ}^{\small\tt oct}}
\newcommand{\tDo}{\tD^{\small\tt oct}}

\newcommand{\Xr}{X^{\small\tt roll}}
\newcommand{\Xo}{X^{\small\tt oct}}

\newcommand{\rs}{r}
\newcommand{\rv}{r'}
\newcommand{\Ad}{{\rm Ad}}

\renewcommand{\v}{{\bf v}} 
\newcommand{\w}{{\bf w}} 

\renewcommand{\j}{{\bf j}} 

\newcommand{\x}{{\bf x}}

\newcommand{\bp}{{\bf p}}
 
 \newcommand{\bQ}{{\bf Q}}

\newcommand{\A}{{\bf A}}

\newcommand{\SL}{\rm{SL}}

\newcommand{\bo}{{\boldsymbol \omega}}
\newcommand{\D}{\mathscr{D}}

\newcommand{\N}{\mathbb{N}}
\newcommand{\Z}{\mathbb{Z}}
\newcommand{\R}{\mathbb{R}}
\renewcommand{\H}{\mathbb{H}}
\renewcommand{\P}{\mathbb{P}}
\newcommand{\tO}{\mathbb{O}}
\newcommand{\RP}{\mathbb{RP}}
\newcommand{\RPt}{\RP^2}
\newcommand{\RPts}{(\RPt)^*}
\newcommand{\Rt}{{\R^3}}
\newcommand{\Rts}{{(\Rt)^*}}
\newcommand{\Rtt}{{\R^{3,3}}}

\newcommand{\SO}{\mathrm{SO}}
\newcommand{\End}{\mathrm{End}}

\newcommand{\GL}{\mathrm{GL}}

\newcommand{\PGL}{\mathrm{PGL}}
\newcommand{\G}{\mathrm{G}_2}
\newcommand{\p}{{\bf b}}
\newcommand{\q}{{\bf A}}

\renewcommand{\b}{{\bf b}}

\newcommand{\e}{{\bf e}}
\renewcommand{\u}{{\bf w}}
\renewcommand{\k}{{\bf k}}
\renewcommand{\i}{{\bf i}}
\renewcommand{\j}{{\bf j}}

\renewcommand{\d}{{\rm d }}
\renewcommand{\Im}{{\rm Im }}
\renewcommand{\Re}{{\rm Re }}
\newcommand{\mn}{\medskip\noindent}

\newcommand{\st}{{\; |\;}}

\newcommand{\so}{\mathfrak{so}}
\renewcommand{\sl}{\mathfrak{sl}}

\newcommand{\g}{\mathfrak{g}}

\newcommand{\zm}{}
\newcommand{\m}{{\zeta}}
\renewcommand{\>}{{\rangle}}
\newcommand{\be}{\begin{equation}}
\newcommand{\ee}{\end{equation}}

\newtheorem{theorem}{Theorem}
\newtheorem{prop}{Proposition}
\newtheorem{lemma}{Lemma}

\newtheorem{cor}{Corollary}
\newtheorem{conj}{Conjecture}
\theoremstyle{definition}
\newtheorem{definition}{Definition}
\newtheorem{example}{Example}
\newtheorem{remark}{Remark}

\begin{document} 
 
\begin{abstract}  
 A pair of planar polygons is  `dancing' if one is inscribed in the other and they satisfy a certain cross-ratio relation at each vertex of the circumscribing polygon. Non-degenerate dancing pairs of closed $n$-gons exist for all $n\geq 6$. Dancing pairs correspond   to trajectories  of  a non-holonomic mechanical system, consisting of a ball rolling, without slipping and twisting, along a polygon drawn on the surface of a ball 3 times larger than the rolling ball. The correspondence stems from  reformulating  both systems as piecewise rigid curves  of  a certain remarkable rank 2 non-integrable distribution  defined on a 5-dimensional quadric in $\RP^6$, introduced  by \'E. Cartan and F. Engel in 1893 in order to define the simple Lie group $\G$.
\end{abstract} 
\maketitle
\tableofcontents


\section{Introduction and statement of results}
This article  connects  two seemingly unrelated themes:
(1)  dancing pairs of polygons in the real projective plane, and
(2)  spherical polygons with trivial rolling monodromy. 
 The connection  is established by relating each of these themes to a third one: (3) piecewise-rigid integral curves of the Cartan-Engel distribution. \\

\def\svgwidth{1\textwidth}
\begingroup%
  \makeatletter%
  \providecommand\color[2][]{%
    \errmessage{(Inkscape) Color is used for the text in Inkscape, but the package 'color.sty' is not loaded}%
    \renewcommand\color[2][]{}%
  }%
  \providecommand\transparent[1]{%
    \errmessage{(Inkscape) Transparency is used (non-zero) for the text in Inkscape, but the package 'transparent.sty' is not loaded}%
    \renewcommand\transparent[1]{}%
  }%
  \providecommand\rotatebox[2]{#2}%
  \newcommand*\fsize{\dimexpr\f@size pt\relax}%
  \newcommand*\lineheight[1]{\fontsize{\fsize}{#1\fsize}\selectfont}%
  \ifx\svgwidth\undefined%
    \setlength{\unitlength}{517.27093073bp}%
    \ifx\svgscale\undefined%
      \relax%
    \else%
      \setlength{\unitlength}{\unitlength * \real{\svgscale}}%
    \fi%
  \else%
    \setlength{\unitlength}{\svgwidth}%
  \fi%
  \global\let\svgwidth\undefined%
  \global\let\svgscale\undefined%
  \makeatother%
  \begin{picture}(1,0.41417855)%
    \lineheight{1}%
    \setlength\tabcolsep{0pt}%
    \put(0,0){\includegraphics[width=\unitlength,page=1]{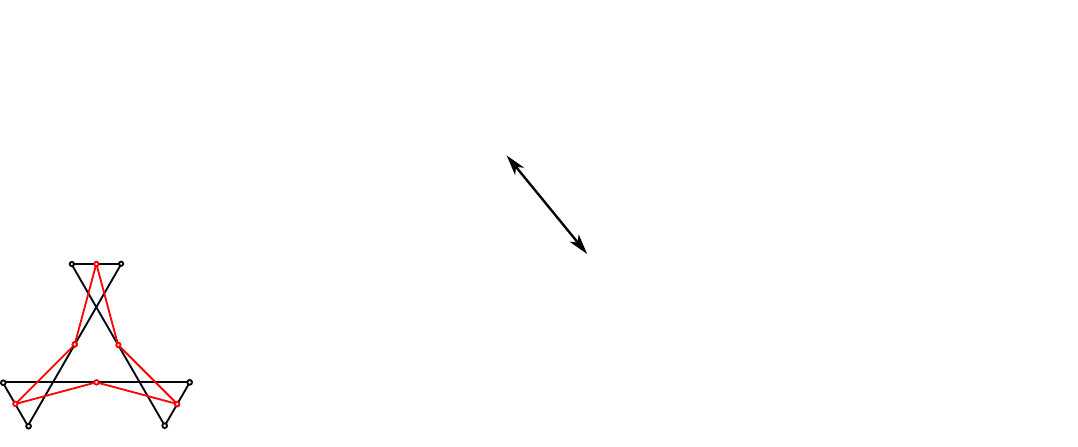}}%
    \put(0.30244902,0.08026885){\color[rgb]{0,0,0}\makebox(0,0)[t]{\lineheight{0}\smash{\begin{tabular}[t]{c}\dd Dancing pairs \\\dd of polygons\end{tabular}}}}%
    \put(0.62769557,0.08103605){\color[rgb]{0,0,0}\makebox(0,0)[t]{\lineheight{0}\smash{\begin{tabular}[t]{c}\dd Spherical polygons\\  \\\dd with trivial \\\dd rolling monodromy\end{tabular}}}}%
    \put(0.45611296,0.38657483){\color[rgb]{0,0,0}\makebox(0,0)[t]{\lineheight{0}\smash{\begin{tabular}[t]{c}\dd Piecewise-rigid\\ \dd integral curves  of the \\\dd Cartan-Engel distribution  \end{tabular}}}}%
    \put(0,0){\includegraphics[width=\unitlength,page=2]{diag2.pdf}}%
    \put(0.2945402,0.12905121){\color[rgb]{0,0,0}\makebox(0,0)[lt]{\lineheight{1.25}\smash{\begin{tabular}[t]{l}1\end{tabular}}}}%
    \put(0,0){\includegraphics[width=\unitlength,page=3]{diag2.pdf}}%
    \put(0.5642075,0.13097206){\color[rgb]{0,0,0}\makebox(0,0)[lt]{\lineheight{1.25}\smash{\begin{tabular}[t]{l}2\end{tabular}}}}%
    \put(0,0){\includegraphics[width=\unitlength,page=4]{diag2.pdf}}%
    \put(0.42904975,0.29190093){\color[rgb]{0,0,0}\makebox(0,0)[lt]{\lineheight{1.25}\smash{\begin{tabular}[t]{l}3\end{tabular}}}}%
    \put(-0.21868477,0.54678527){\color[rgb]{0.00784314,0,0}\transparent{0.99607801}\makebox(0,0)[lt]{\lineheight{1.25}\smash{\begin{tabular}[t]{l} \end{tabular}}}}%
    \put(0,0){\includegraphics[width=\unitlength,page=5]{diag2.pdf}}%
  \end{picture}%
\endgroup%

\bigskip

In this introductory  section we describe themes (1) and (2) and state the relation between them, see Theorem \ref{thm:main}. In the next section (Section 2) we describe theme (3) and   the relations $(1)\leftrightarrow (3)$ and $(2)\leftrightarrow(3)$. Themes (2) and (3) and their relation were previously  studied \cite{BrH,BM,BH}; theme  (1) and its relation to (3) is  new and can be thought of as a discrete version of  our previous work with P. Nurowski in \cite{BHN}.  The main new technical result of the paper, establishing the relation  $(1)\leftrightarrow (3)$, is Theorem \ref{thm:horpol} of Section 2 (proved in Section \ref{sec:horpol}).

In Section 3 we give proofs of the theorems stated in the first two sections. In Appendix A we give explicit formulas for the infinitesimal  action of the group $\G$ on the various configuration spaces appearing in this article
(this action is due to a well-known symmetry property of the Cartan-Engel distribution, see Section \ref{ss:oct} below). In Appendix B we give explicit coordinate formulas for the `rolling distribution' which models the rolling balls system of theme (2).


 In general, we tried to make the article as self-contained as possible, without assuming the reader's familiarity with any of the mentioned themes.

\subsection{Dancing pairs of  polygons} These are inscribed pairs of planar polygons,   satisfying a certain system of scalar equations, one equation  for each vertex of the circumscribing polygon, involving cross-ratios of neighboring vertices. The precise definition is as follows. 

\mn\begin{minipage}{.65\textwidth}
\hspace{1.5em}Consider a pair of  polygons in the projective plane $\RPt$, open or closed, the first with $n$ vertices $A_1, A_2, \ldots , A_n$ and the second with  $n$  edges $b_1, b_2, \ldots, b_{n}$, $n\geq 2$.  The second polygon is {\em inscribed} in the first  if each vertex $B_i:=b_ib_{i+1}$  (the intersection of $b_i$ with $b_{i+1}$) lies on the edge $a_i:=A_iA_{i+1}$  (the    line   through  $A_{i}$ and $A_{i+1}$). If the polygons  are open then $i=1,\ldots, n-1$ and  if they are closed  then $i=1,\ldots,n$ and indices   are considered mod $n$.  
\mn
\end{minipage}
\quad 
\begin{minipage}{.35\textwidth}
\centering
\def\svgwidth{.9\textwidth}
\begingroup%
  \makeatletter%
  \providecommand\color[2][]{%
    \errmessage{(Inkscape) Color is used for the text in Inkscape, but the package 'color.sty' is not loaded}%
    \renewcommand\color[2][]{}%
  }%
  \providecommand\transparent[1]{%
    \errmessage{(Inkscape) Transparency is used (non-zero) for the text in Inkscape, but the package 'transparent.sty' is not loaded}%
    \renewcommand\transparent[1]{}%
  }%
  \providecommand\rotatebox[2]{#2}%
  \newcommand*\fsize{\dimexpr\f@size pt\relax}%
  \newcommand*\lineheight[1]{\fontsize{\fsize}{#1\fsize}\selectfont}%
  \ifx\svgwidth\undefined%
    \setlength{\unitlength}{257.82377468bp}%
    \ifx\svgscale\undefined%
      \relax%
    \else%
      \setlength{\unitlength}{\unitlength * \real{\svgscale}}%
    \fi%
  \else%
    \setlength{\unitlength}{\svgwidth}%
  \fi%
  \global\let\svgwidth\undefined%
  \global\let\svgscale\undefined%
  \makeatother%
  \begin{picture}(1,0.73211396)%
    \lineheight{1}%
    \setlength\tabcolsep{0pt}%
    \put(0,0){\includegraphics[width=\unitlength,page=1]{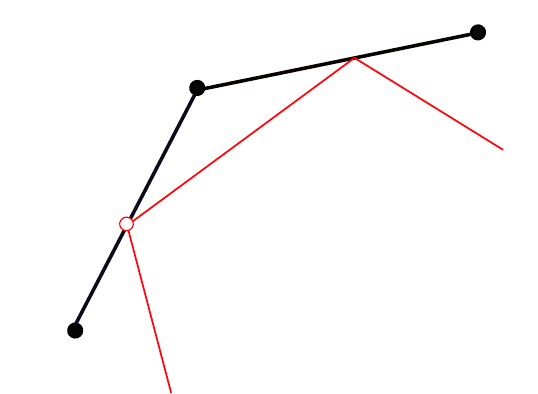}}%
    \put(0.3083034,0.13423565){\color[rgb]{0,0,0}\makebox(0,0)[lt]{\lineheight{0}\smash{\begin{tabular}[t]{l}\s$b_{i}$\end{tabular}}}}%
    \put(-0.00230533,0.07544795){\color[rgb]{0,0,0}\makebox(0,0)[lt]{\lineheight{0}\smash{\begin{tabular}[t]{l}\s$A_{i}$\end{tabular}}}}%
    \put(-0.1591986,0.3664076){\color[rgb]{0,0,0}\makebox(0,0)[lt]{\lineheight{0}\smash{\begin{tabular}[t]{l} \end{tabular}}}}%
    \put(0.26801973,0.61344829){\color[rgb]{0,0,0}\makebox(0,0)[lt]{\lineheight{0}\smash{\begin{tabular}[t]{l}\s$A_{i+1}$\end{tabular}}}}%
    \put(0.13459822,0.37029735){\color[rgb]{0,0,0}\makebox(0,0)[lt]{\lineheight{0}\smash{\begin{tabular}[t]{l}\s$B_{i}$\end{tabular}}}}%
    \put(0.58144899,0.67436542){\color[rgb]{0,0,0}\makebox(0,0)[lt]{\lineheight{0}\smash{\begin{tabular}[t]{l}\s$B_{i+1}$\end{tabular}}}}%
    \put(0.85016698,0.70942227){\color[rgb]{0,0,0}\makebox(0,0)[lt]{\lineheight{0}\smash{\begin{tabular}[t]{l}\s$A_{i+2}$\end{tabular}}}}%
    \put(0.70760359,0.46503281){\color[rgb]{0,0,0}\makebox(0,0)[lt]{\lineheight{0}\smash{\begin{tabular}[t]{l}\s$b_{i+2}$\end{tabular}}}}%
    \put(0.42773336,0.37891355){\color[rgb]{0,0,0}\makebox(0,0)[lt]{\lineheight{0}\smash{\begin{tabular}[t]{l}\s$b_{i+1}$\end{tabular}}}}%
    \put(0,0){\includegraphics[width=\unitlength,page=2]{inscribed.pdf}}%
  \end{picture}%
\endgroup%

\end{minipage}

\begin{definition}[Dancing pairs]\label{def:dp} An inscribed  pair of  polygons, as above,  is a  {\em dancing pair}   if 
\begin{equation}\label{eq:dp}
[A_{i+1}, B_{i}, A_{i},D]+[ A_{i+1},B_{i+1}, A_{i+2},C]=0, 
\end{equation}
where 
$
C:=b_i a_{i+1},$ 
$D:=a_ib_{i+2}$, $i=1,\ldots,n-2$ for open polygons and  $i=1,\ldots,n$  for closed polygons, in which case indices are  considered  mod $n$. See Figure \ref{fig:dp}.

Note that  for  $n=2$ any inscribed (open) pair is automatically dancing.

\end{definition}

\begin{figure}[h]
\centering
\def\svgwidth{1\textwidth}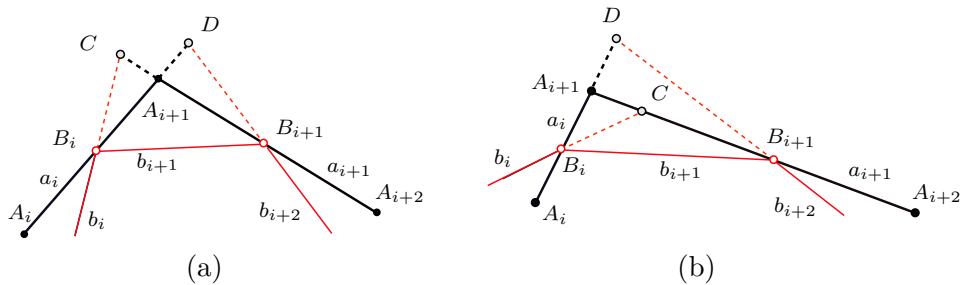
\caption{The dancing condition, see Equation \eqref{eq:dp}. Actually, (a)  does not satisfy the dancing condition. Can you see why? See Remark  \ref{rem:dp}. }
\label{fig:dp}
\end{figure}

 The cross-ratio 
in Equation \eqref{eq:dp} is defined for 4 collinear points  in $\RPt$  by 
\be\label{eq:cp} [A_1, A_2, A_3, A_4]:= {(x_1-x_3 )(x_2-x_4)\over (x_1-x_4)(x_2-x_3 )},
\ee
 where $x_i$ are the coordinates of $A_i$ with respect   to some affine coordinate  along the containing  line. 
 
 As is well known, the cross-ratio of 4 colinear points in $\RPt$ is projectively invariant; that is,  invariant under the action  of the standard action of the projective group $\PGL_3(\R)$ on $\RPt$. 
 
\begin{remark}\label{rem:dp}
  Figure \ref{fig:dp}(a), although pleasantly symmetric  and depicts correctly 
 the definition of $C,D$  in Equation \eqref{eq:dp}, is  not of  a dancing pair, since  both summands on the left hand side of Equation \eqref{eq:dp} in this figure are positive. The reason is  that, as can be easily verified,  the cross-ratio $[A_1, A_2, A_3,A_4]$  is positive if and only if  $A_1,A_2$ does not `separate' $A_3,A_4$.  See Figure \ref{fig:scr}.\end{remark}
 
 \begin{figure}[h]\centering\def\svgwidth{.8\textwidth}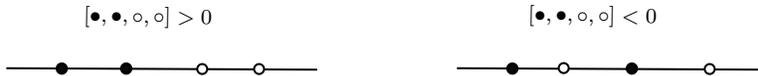\caption{The sign of the cross-ratio.  }\label{fig:scr}\end{figure}

The term `dancing' in Definition \ref{def:dp} is taken from  \cite{BHN}. The present article can be thought of as a discrete version  of it. 

For Definition \ref{def:dp} to make sense, one needs to make some genericity assumptions on the pair of polygons. Let us spell them out:

\begin{definition}
A pair of inscribed polygons in $\RPt$, the first with vertices $A_1,\dots,A_n$ and  the second with edges $b_1, \ldots, b_n$,  is  {\em non-degenerate} if 
\begin{itemize}
\item Each  vertex $A_i$ of the first polygon does not lie on  the `opposite' side $b_i$ of the second polygon. 
\item Each three consecutive  vertices of the first polygon are non-collinear and each three consecutive edges of the second polygon are non-concurrent. 
\end{itemize}
\end{definition}

\paragraph{Existence of closed dancing pairs.} It is easy to produce examples of non-degenerate dancing pairs of {\em open} polygons for all $n\geq 2$. For example, one picks arbitrary (generic) $A_1,\ldots, A_n, b_1,$ then arbitrary (generic) $b_2$ incident to   $b_1( A_1A_2)$, after which $b_3,...., b_n$ are determined recursively by equation  \eqref{eq:dp} for $i=1,\ldots, n-2.$ This gives a total of $(2n+3)$-parameter family of open dancing pairs.  When the polygons are closed one needs to  add Equation \eqref{eq:dp} also for  $i=n-1, n$ and the equation $b_{n+1}=b_1$, reducing the number of parameters to $2n$. (One can also make this naive parameter count by considering that pairs of inscribed $n$-gons depend on $3n$ parameters and the $n$ dancing conditions reduce this to $2n$.)  

However,  by examining the signs of the summands in  Equation \eqref{eq:dp},  one can see easily that there are no dancing pairs of triangles. More generally, we have the following result, which will be proved in later  sections.  

\begin{theorem}\label{thm:intro1}There are non-degenerate dancing pairs of closed $n$-gons  if and only if  $n\geq 6$. 
\end{theorem}

Figure \ref{fig:hex}  shows an example of a non-degenerate dancing pair of closed hexagons. 
\begin{figure}[H]
\centering
\includegraphics[width=.3\textwidth]{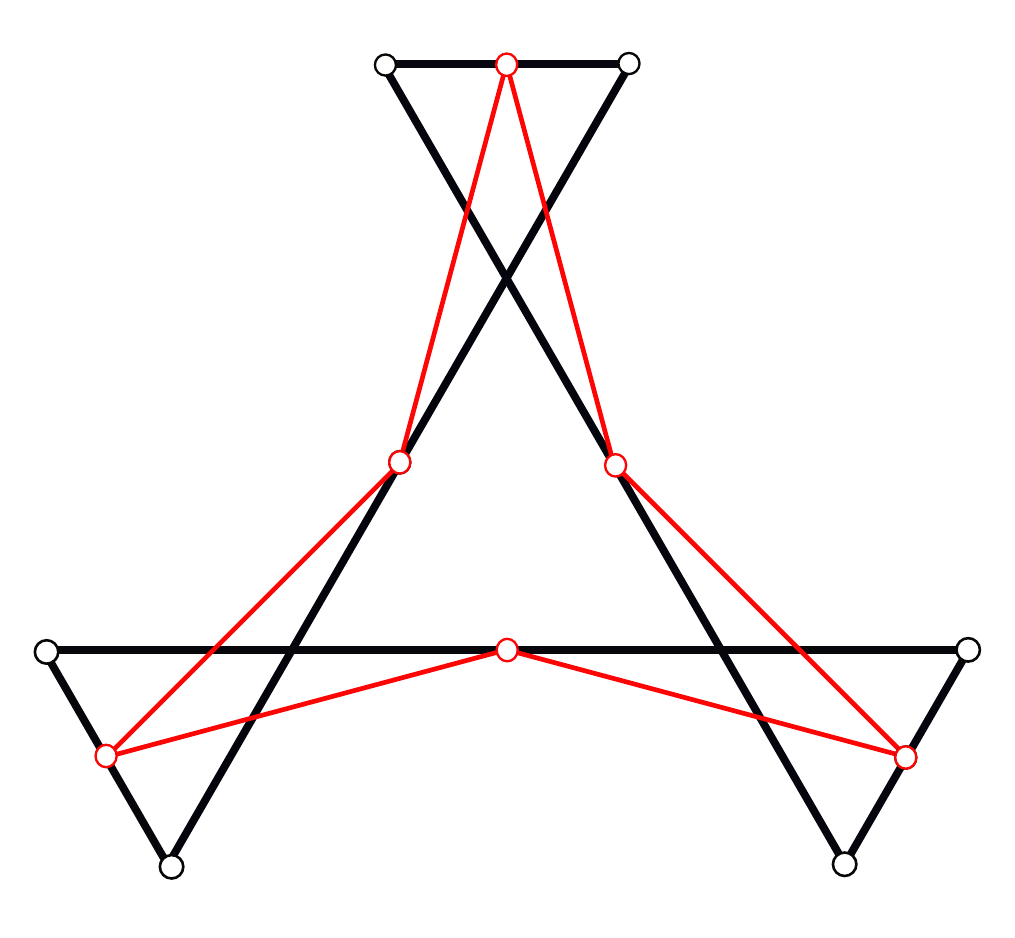}
\caption{A dancing pair of hexagons.}\label{fig:hex}
\end{figure}
\paragraph{Symmetries.} Clearly, by the projective invariance of the cross-ratio appearing in Definition \ref{def:dp}, the 8-dimensional projective group $\PGL_3(\R)$ acts on the space of closed dancing pairs, so one expects, for $n$ big enough,  a $(2n-8)$-parameter family of projective congruence classes (e.g., a 4-parameter family of non-trivial deformations of the example of Figure \ref{fig:hex}). A somewhat surprising construction in this paper   (see Section \ref{ss:oct}) is an effective (local) action of the  14-dimensional exceptional simple non-compact Lie group $\G$ on the space of closed dancing pairs, so one expects a $2n-14$ family of such congruence classes for $n\geq 8$. For  $n=6,7$ one expects a discrete family of $\G$-congruence classes. 

\begin{conj} For $n=6,7$, all closed dancing pairs of $n$-gons  are $\G$ congruent. 
\end{conj}

\subsection{Spherical polygons, rolling balls and monodromy} 

The second theme of this article is a well-known non-holonomic mechanical system, see  
\cite{Ag, AN, BM, BH, BrH}.  Consider a  `stationary' round sphere of radius $\rs$ in Euclidean $\R^3$, on which  a closed  oriented   polygonal path $\Gamma$ is drawn (its edges are arcs of great circles). We impose a {\em non-degeneracy} condition on $\Gamma$: no three consecutive vertices are `collinear', i.e., lie on the same great circle. Note that this implies that no two consecutive vertices are antipodal.

Next take another  sphere, a `moving' sphere, of radius $\rv$,  place it  outside the stationary sphere, touching it at one of the  vertices  of $\Gamma$, then roll it along $\Gamma$ without slipping  or twisting, see Figure \ref{fig:spheres}. (A formal definition of these terms will be given later in Section \ref{ss:rol}.)

\noindent\begin{minipage}{.65\textwidth}
\hspace{1.5em}As we roll the moving sphere  along $\Gamma$,  it rotates about its center. After going  once around $\Gamma$, the moving sphere returns to the initial vertex, but  possibly with a  different orientation, given by an element of  the orthogonal group $\SO_3$,
called the {\em rolling monodromy} of $\Gamma$. Put differently, as we roll the moving sphere along $\Gamma$, its rotation about its center defines a curve in  $\SO_3$, starting at the identity, whose other endpoint  is the rolling monodromy of $\Gamma$. 
The curve in $\SO_3$  can be lifted to a curve in the universal double-cover $S^3\to\SO_3$, starting at $1\in S^3$ (we are thinking of $S^3$ as the sphere of unit quaternions), whose  other end point is the {\em lifted rolling monodromy} of $\Gamma.$ 
\end{minipage}
\quad
\begin{minipage}{.35\textwidth}
\centering
\def\svgwidth{.8\textwidth}
\begingroup%
  \makeatletter%
  \providecommand\color[2][]{%
    \errmessage{(Inkscape) Color is used for the text in Inkscape, but the package 'color.sty' is not loaded}%
    \renewcommand\color[2][]{}%
  }%
  \providecommand\transparent[1]{%
    \errmessage{(Inkscape) Transparency is used (non-zero) for the text in Inkscape, but the package 'transparent.sty' is not loaded}%
    \renewcommand\transparent[1]{}%
  }%
  \providecommand\rotatebox[2]{#2}%
  \newcommand*\fsize{\dimexpr\f@size pt\relax}%
  \newcommand*\lineheight[1]{\fontsize{\fsize}{#1\fsize}\selectfont}%
  \ifx\svgwidth\undefined%
    \setlength{\unitlength}{227.15327454bp}%
    \ifx\svgscale\undefined%
      \relax%
    \else%
      \setlength{\unitlength}{\unitlength * \real{\svgscale}}%
    \fi%
  \else%
    \setlength{\unitlength}{\svgwidth}%
  \fi%
  \global\let\svgwidth\undefined%
  \global\let\svgscale\undefined%
  \makeatother%
  \begin{picture}(1,1.29976116)%
    \lineheight{1}%
    \setlength\tabcolsep{0pt}%
    \put(0,0){\includegraphics[width=\unitlength,page=1]{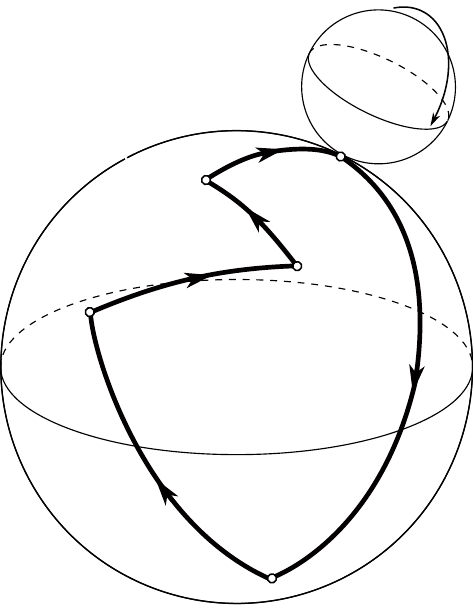}}%
    \put(0.22422927,0.23521517){\color[rgb]{0,0,0}\makebox(0,0)[lt]{\lineheight{1.25}\smash{\begin{tabular}[t]{l}$\Gamma$\end{tabular}}}}%
  \end{picture}%
\endgroup%

\captionof{figure}{\ }
\label{fig:spheres}
\end{minipage}

\begin{definition}
The rolling monodromy of $\Gamma$ is {\em trivial} if it is the identity element in $\SO_3$. Similarly, for the lifted monodromy. 
\end{definition}

 In other words, the rolling monodromy of $\Gamma$ is trivial if the associated curve in $\SO_3$ is {\em closed}, and the  lifted monodromy is trivial if the lifted curve in $S^3$ is closed, which is the  same as requiring the curve in $\SO_3$ to be closed and  {\em null homotopic}.

Clearly, if we pick another initial point on $\Gamma$ then the (lifted) rolling monodromy differs by conjugation by an element in $\SO_3$,  which does not affect its triviality.
 
 Note that the (lifted) rolling monodromy, and in particular its triviality, does depend on the radius ratio of the two spheres, $\rho:=\rs/\rv.$ 

\begin{example}\label{ex:eq} Let $\Gamma$ be the  equator of the stationary sphere (a horizontal  great circle). Rolling the moving sphere  once around  $\Gamma$  results in it being rotated  $\rho+1$  times  about the  vertical axis through its center (this is elementary, but rather counterintuitive; see this animation \cite{Geo}).  Lifted to $S^3$, we get a path going $(\rho+1)/2$ times around a great circle of $S^3$.  Thus the rolling monodromy of such a $\Gamma$ is trivial if and only if $\rho$ is an integer, and the lifted monodromy is trivial if and only if  $\rho$ is an {\em odd integer}. 
\end{example}

\noindent\begin{minipage}{.65\textwidth}
\begin{example}
(We used this   example  to produce Figure \ref{fig:hex}.)  Let $\Gamma$ be an `octant', i.e., an equilateral  spherical triangle, with side length a quarter of a great circle.  For $\rho=3$, each of its sides, by the previous example,  results in a lifted monodromy of $-1\in S^3$,  adding up to a total lifted  monodromy of $(-1)^3=-1.$ Thus rolling {\em twice} around $\Gamma$ results in a  trivial lifted  monodromy. 
\end{example}
\end{minipage}
\quad
\begin{minipage}{.35\textwidth}
\centering
\includegraphics[width=.8\textwidth]{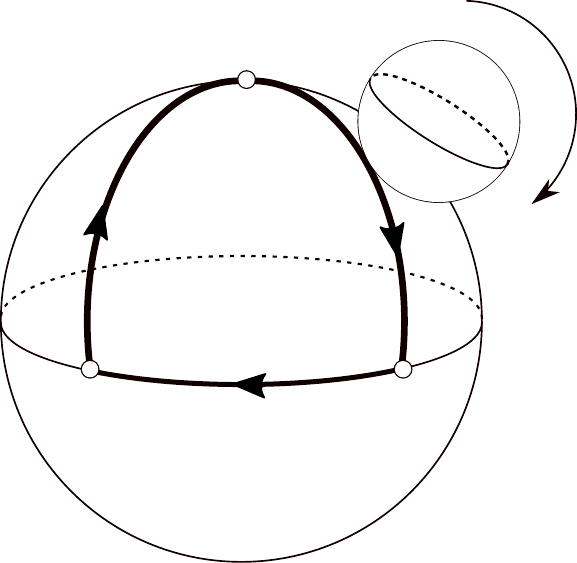}
\captionof{figure}{\  }
\label{fig:octant}
\end{minipage}

\mn 

We next present  an infinite family of examples of non-degenerate regular spherical $n$-gons with trivial lifted rolling monodromy, valid for $\rho=3$ and  all $n\geq 6$.

\paragraph{Regular spherical polygons with trivial (lifted) monodromy.} 
Consider  a closed non-degenerate regular spherical $n$-gon $\Gamma$, contained in the (open) northern hemisphere of the stationary sphere, and whose vertices lie on a circle of latitude of radius $\phi\in(0,\pi/2)$ (the `colatitude').

Let $w\in\N$  be the winding number  of  $\Gamma$ about the north pole. Note that by the non-degeneracy assumption, $\Gamma$ does not pass through the poles, so $w$ is well defined. The  `rotation  angle' between  two successive vertices is then $\theta:=2\pi w/n,$ where $0<\theta<\pi$, i.e., $0<w<n/2$. 
See Figure \ref{fig:reg}. 

\begin{figure}[h]
\centering
\def\svgwidth{.4\textwidth}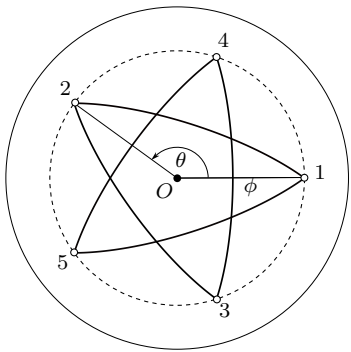
\caption{A regular spherical pentagon, projected onto the $xy$ plane, with winding number $w=2$ about the north pole $O$ and rotation angle of $
\theta=4\pi/5$.
}\label{fig:reg}
\end{figure}

We ask: for which  $(n,w,\phi, \rho)$  is the (lifted) rolling monodromy of  $\Gamma$ trivial?  
We shall answer this question only  for $\rho=3$, the case that interests us here.

\begin {theorem}\label{thm:regularpolys}Let $\Gamma$ be a non-degenerate  regular spherical $n$-gon on the stationary sphere, with winding number $w$ about its center  and circumscribing circle of radius $\phi$. Furthermore, we assume that  $n\geq 3$, $0<w<n/2$ and $\phi\in(0, \pi/2)$, as described above. Then 
\begin{enumerate}[{\rm (a)}]
\item $\Gamma$ has  trivial rolling monodromy for $\rho=3$ if and only if there exists a (necessarily unique) integer $w'$  in the range  $w<w'<n$ such that 
\be\label{eq:roll}
\cos\left({\pi  w'\over  n}\right)=\cos\left({\pi  w\over  n}\right)
\left[1-4\sin^2\left({\pi  w\over n}\right)\sin^2\phi\right].
\ee


\item The  lifted rolling monodromy of such $\Gamma$  is trivial if and only if  $w'\equiv w$ {\rm (mod 2)}. 
\item $w'$ in Equation \eqref{eq:roll} is the winding number of the closed regular polygon traced on the moving sphere as it rolls along $\Gamma$.

\item There are solutions to Equation \eqref{eq:roll} with  $w'\equiv w$ {\rm (mod 2)}  if and only if  $n\geq 6$. In fact, there is a solution for $w=2, w'=4$ and all $n\geq 6.$ 
\end{enumerate}
\end {theorem}

See Section \ref{sec:reg} for a proof. Figure \ref{fig:regs} shows some examples with $\leq 10$  vertices.

\begin{figure}[H]
\centering
\begin{tabular}{ccccc}
\includegraphics[width=.175\textwidth]{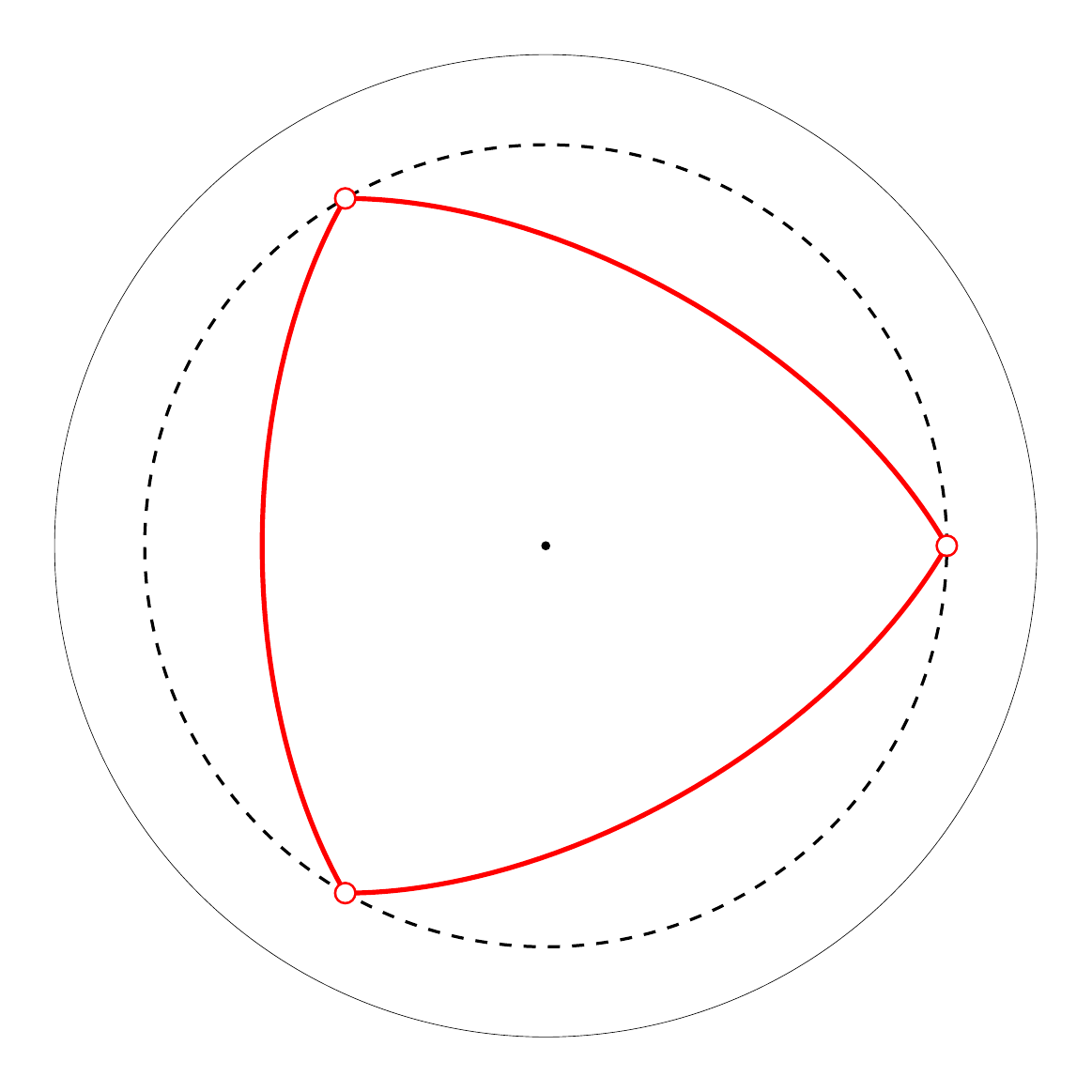}&
\includegraphics[width=.175\textwidth]{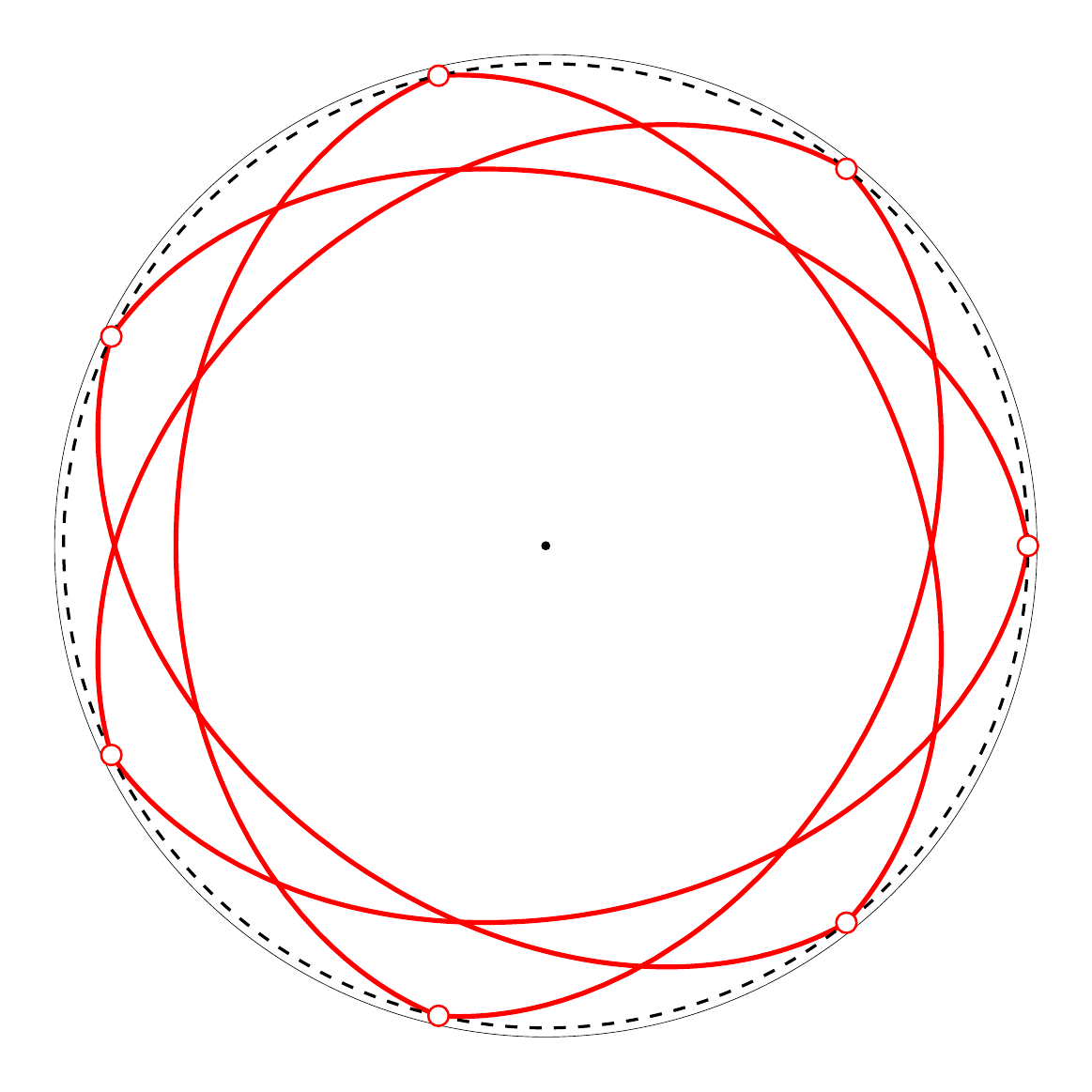}&
\includegraphics[width=.175\textwidth]{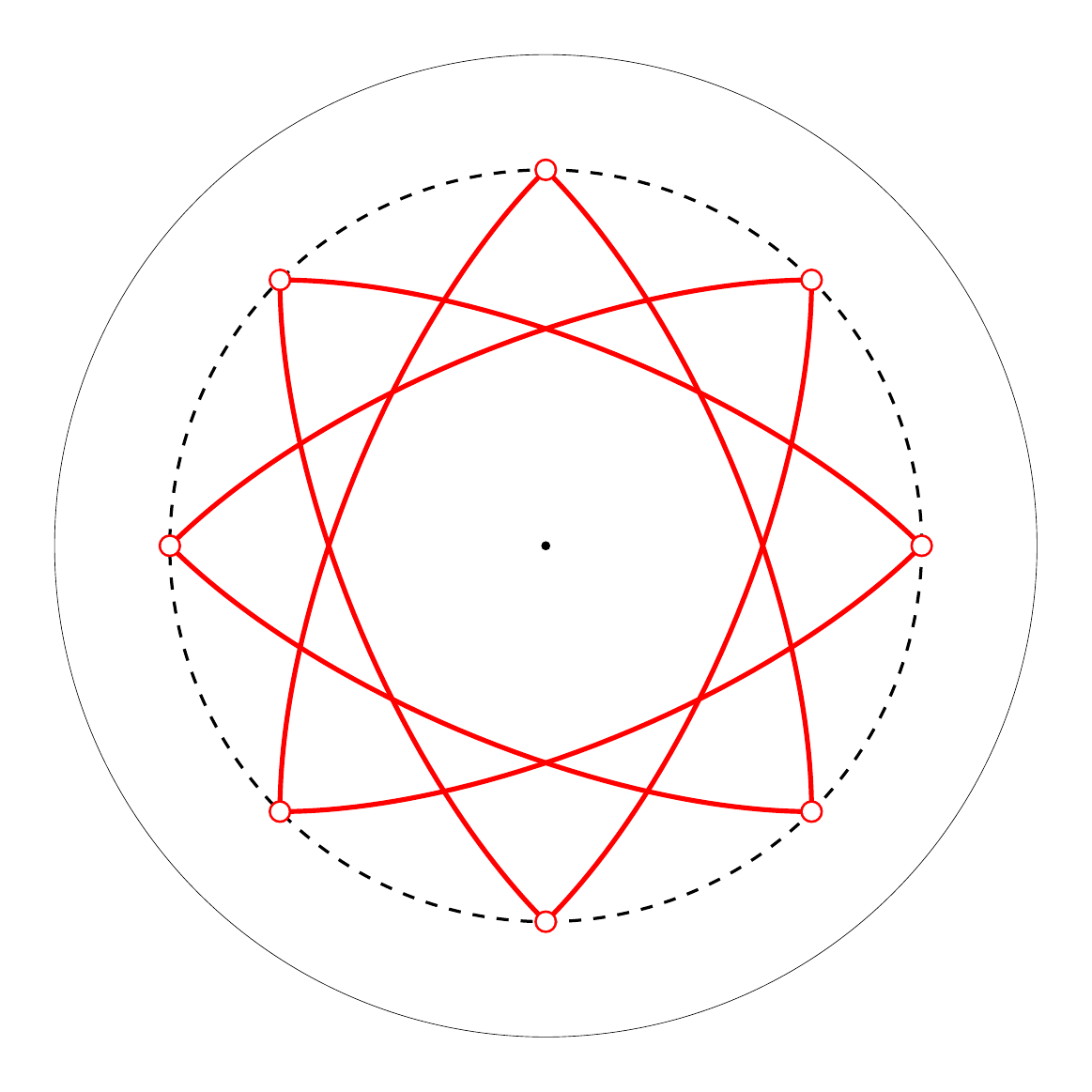}&
\includegraphics[width=.175\textwidth]{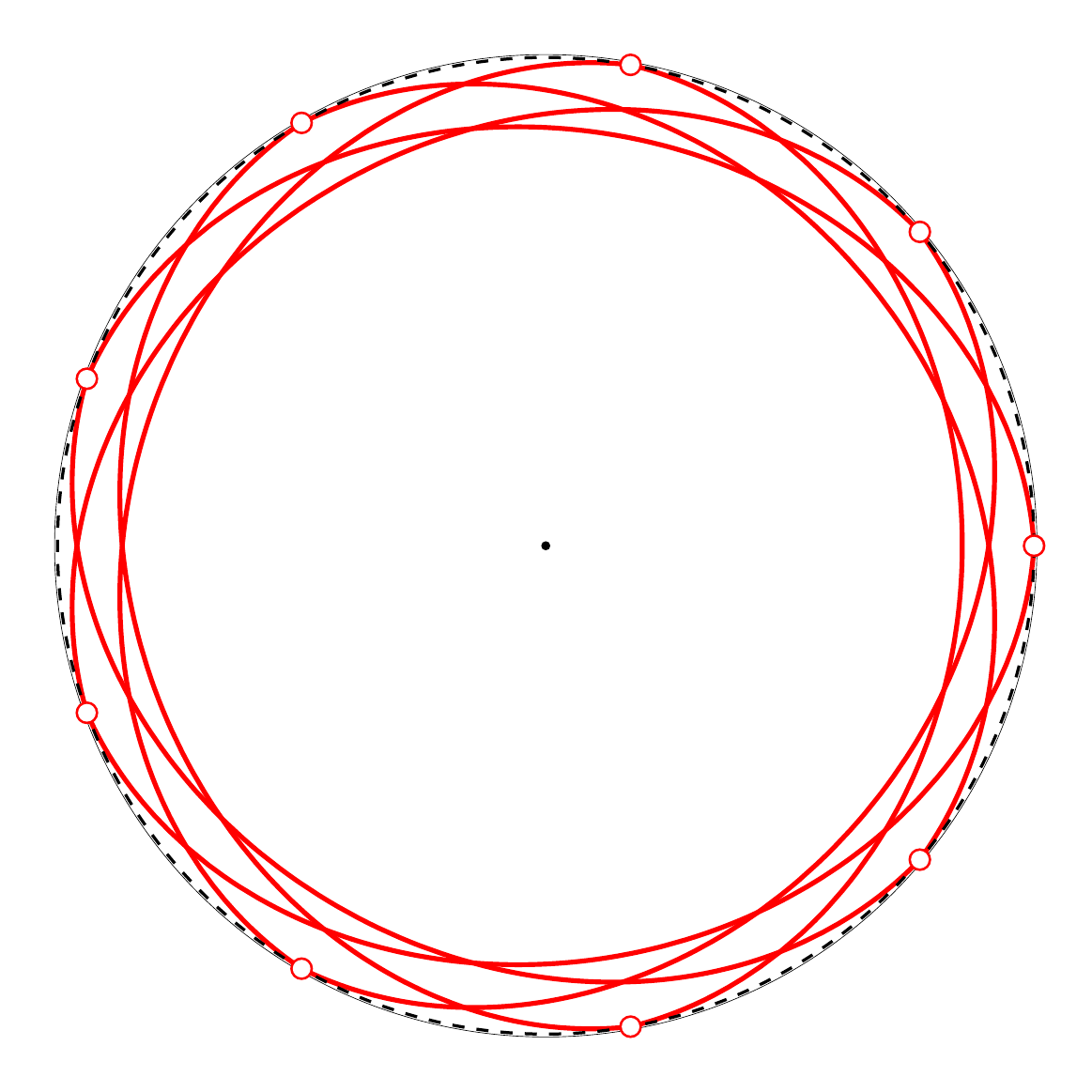}&
\includegraphics[width=.175\textwidth]{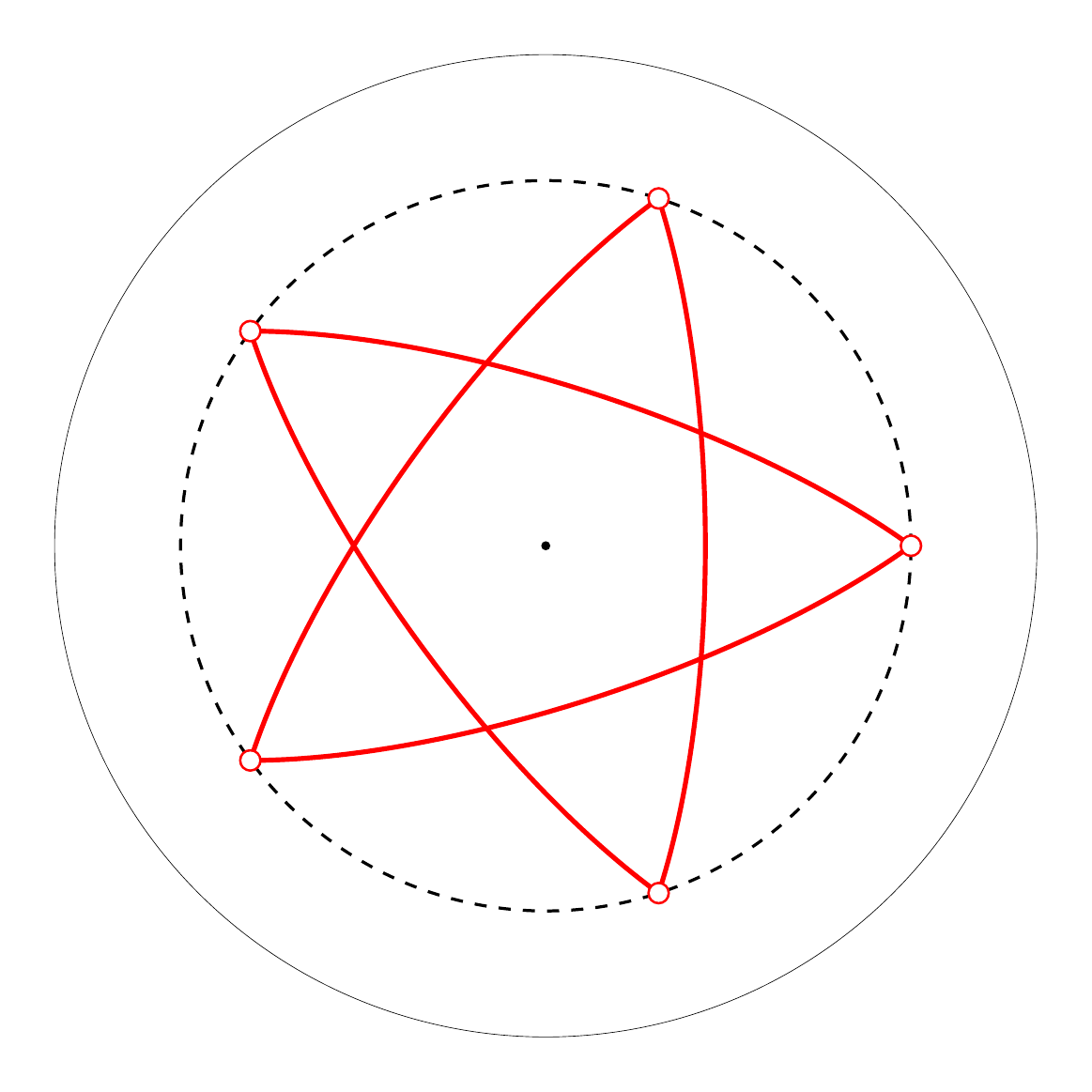}\\
$(6,2,4)$ & $(7,2,4)$  & $(8,3,5)$  & $(9,2,4)$  & $(10,4,6)$ \\
&&&& 
\end{tabular}
\caption{Regular spherical $n$-gon, $n=6,7,8,9,10$,  with trivial lifted rolling monodromy for $\rho=3$, projected to the $xy$ plane. The triple of numbers below each figure is $(n,w,w')$. }\label{fig:regs}
\end{figure}


\begin{cor}\label{cor:regularpolys}For each $n\geq 6$ there exists a closed non-degenerate spherical $n$-gon with lifted trivial monodromy for $\rho=3$.
\end{cor}

\begin{remark} Theorem \ref{thm:regularpolys} does not give an explicit list of  all regular polygons with trivial lifted monodromy for $\rho=3$, since it only reduces the question to Equation \eqref{eq:roll}, without fully  solving it. We will not dwell on solving this equation completely, since the solutions given in part (d) of Theorem \ref{thm:regularpolys} are sufficient for this   article. However,  some numerical experiments indicate the following\footnote{We thank Carlos Licea, an undergraduate physics student of the University of Guanajuato, for helping with these experiments.}: 
\begin{itemize}
\item  If $(n,w,w')$ is an admissible triple (i.e., one can find $\phi\in(0,\pi/2)$ that solves the equation, with $0<w<w'<n$, $w<n/2$ and $w\equiv w'$ mod 2) then so are all  $(n', w,w')$ with  $n'>n$. 

\item Let us say that an  admissible triple $(n,w,w')$ is {\em minimal} if $(n',w,w')$ is not admissible for $n'<n$. Then, for a fixed $n$,   the minimal admissible triples  and their number $m$ are as follows, with $k\geq 2$ and $j=1,2,\ldots, m$  in all 3 cases: 
\begin{itemize}
\item If $n=3k$ then $w=k+j-1, w'=n-3j+1$, $m=[k/2]$.
\item if $n=3k+1$ then $w=k+j, w'=n-3j -1$, $m=[(k-1)/2].$

\item if $n=3k+2$ then $w=k+j, w'=n-3j$, $m=[k/2].$

 \end{itemize}
\item In particular, there are minimal admissible triples  for  all $n\geq 6,$ except $n=7$. The minimal triples for $n\leq 12$ are  
 $ (6,2, 4 ),$ $  (8,3 , 5),$ $  (9, 3 , 7), $ $(10,4, 6),$ $ (11,4, 8), $ $(12, 4, 10),$ $ (12,5, 7).$
 \end{itemize}
\end{remark}

\subsection{The main theorem}
Our  main result  relates dancing pairs of closed polygons with spherical polygons  with lifted trivial rolling monodromy for radius ratio  $\rho=3$. What is special about $\rho=3$ (as well as $1/3$) will be explained in the next section, once we interpret sphere rolling in terms  of the Cartan-Engel distribution. For the moment, to state the relation, we need to add a definition.

 Note that  a non-degenerate  spherical polygon $\Gamma$ is not determined by its vertices: for each pair of successive  vertices there are infinitely many `edges'  (directed arcs of great circles) connecting them; two of which are `simple', i.e., non self-intersecting, complementary arcs of the great circle containing the two points, and the rest wrap around this circle an arbitrary number of times. 
 
 \begin{definition}\label{def:equiv}Two closed spherical polygons are {\em equivalent} if they have the same set of vertices,   and/or  some of the vertices are replaced by  their antipodes.
 \end{definition}
In other words, equivalence classes of  non-degenerate closed spherical $n$-gons 
 are given by non-degenerate ordered sets of $n$ points in  $S^2/\pm 1\cong \RP^2$. 
 
 Note that, by Example \ref{ex:eq},  triviality of the lifted rolling monodromy of $\Gamma$ for $\rho=3$ is preserved by this equivalence relation (see Lemma \ref{lemma:lifts} below for further details).

 The initial placement of the moving sphere at one of the vertices of $\Gamma$ is  given by an element of $\SO_3$, the orientation  of the moving sphere with respect to some fixed reference orientation. 
The lifted  path in $S^3$, describing the motion of the moving sphere along $\Gamma$, is thus determined by an arbitrary initial element $q\in S^3$   chosen `above' the  initial vertex in  $\Gamma$.


\begin{theorem}[main]\label{thm:main}There is a bijection between non-degenerate  dancing pairs of closed $n$-gons in $\RP^2$ and 
generic\footnote{The  precise meaning of `generic' will be given in the next section, once we describe the bijection in detail.} 
pairs $([\Gamma], q)$, where  $q\in S^3$ and $[\Gamma]$ is an equivalence class of non-degenerate closed spherical $n$-gons  with trivial lifted  rolling monodromy for $\rho=3$. 
\end{theorem}

 For example, to the spherical octant of Figure \ref{fig:octant}, traversed twice, corresponds the dancing pair of hexagons of Figure \ref{fig:hex}.
Similarly, the regular spherical polygons with trivial lifted monodromy of Theorem \ref{thm:regularpolys} correspond to 
examples of dancing pairs of $n$-gons for all $n\geq 6$ (one half of Theorem \ref{thm:intro1}).

\section{The Cartan-Engel distribution}
The  relation between dancing pairs and rolling balls, as  indicated  in Theorem  \ref{thm:main},  is based on modeling both problems by the same remarkable geometric object:  a  certain non-integrable rank 2 distribution    on a 5-manifold,  introduced by Eli\'e Cartan and Friedrich Engel in 1893 (seemingly independently),  in  order to define the simple exceptional 14-dimensional Lie group $G_2$ \cite{Ca1, E}. The subject  has since been studied extensively by many authors (including ourselves,  see \cite{BM, BHN}). Here is a quick review of the basic properties relevant here. 

Let  $\D$ be a rank 2 distribution on a 5-manifold $Q$, i.e., a rank 2 subbundle of $TQ$. It is said to be a \ttf-distribution if  $[[\D,\D],\D]=TQ$; that is,  for any  local framing  $X_1,X_2$ of $\D$ (i.e., two everywhere independent local sections of $\D$), let $X_3:=[X_1,X_2], X_4:=[X_1,X_3],  X_5:=[X_2,X_3],$ then $X_1,\ldots, X_5$ is a local framing of $TQ$. 
In fact, a generic rank 2 distibution on a 5-manifold is \ttf\  but typically no two are diffeomorphic, even locally. A (local) symmetry of a distribution $\D$ on a manifold $Q$ is a (local) self-diffeomorphism of $Q$ preserving $\D$.  

In his fundamental paper on the subject \cite{Ca2} Cartan showed that the maximal dimension of the local symmetry group of a \ttf-distribution is 14, in which case the local symmetry group is the simple non-compact Lie group $\G$ and the distribution is called {\em flat}. Cartan also showed in \cite{Ca2} that all flat \ttf-distributions are locally diffeomorphic. 
Cartan and Engel gave in 1893 explicit formulas for such a distribution on $\R^5$ (their formulas in \cite{Ca1, E}  look  similar to our formula of Equation \eqref{eq:dist} below). We thus call a flat \ttf-distribution a {\em Cartan-Engel distribution}, or, by a slight abuse of terminology, {\em the} Cartan-Engel distribution. 

 Another result of Cartan in \cite{Ca2} is {\em Cartan's submaximal symmetry statement}:  the local symmetry group of a non-flat \ttf-distribution is at most 7-dimensional. 
 
 A much more recent  general  result on \ttf-distributions, by Bryant and Hsu \cite{BrH} (see also the last paragraph of \cite{Br}), is the existence of  {\em rigid} (or `abnormal') integral curves: any small enough segment of such a curve admits  no deformations  within the class of integral curves with the same end-points. In fact, the rigid curves of a \ttf-distribution $\D$  form (locally) a 5-parameter family  of $\D$-integral curves,  a unique such curve passing through  a given point in $Q$ in a given  direction at the point tangent to $\D$. 

We next describe briefly 3 models of the Cartan-Engel distribution that are used in  this paper and their interrelations. 

\subsection{ First model: dancing pairs}

Let $\Rtt:=\Rt\times \Rts$, a 6-dimensional real vector space equipped with the quadratic form $(\A, \b)\mapsto \b\A$ of signature $(3,3)$  (we are thinking of $\Rt$ as column  vectors and $\Rts$ as row vectors). 
We use  the standard  volume forms   $vol, vol^*$  on $\Rt,\Rts$ (respectively)  to define  `vector products' $\Rt\times\Rt\to \Rts$, $\Rts\times\Rts\to \Rt$ by 
\be\label{eq:prod}\A_1\times \A_2:=vol(\A_1, \A_2,\,  \cdot\, ) ,\ \  \b_1\times \b_2=vol^*(\b_1, \b_2,\, \cdot\, ).
\ee


Denote the projection $\Rt\setminus 0\to\RPt$  by $\A\mapsto [\A]$, and similarly for $\Rts$. That is, we are thinking of $\Rt$ and $\Rts$ as  homogeneous coordinates of points and lines in  $\RPt$ (respectively).  


Next, let  $\Qd\subset \Rtt$ be the 5-dimensional affine quadric 
$$\Qd:=\{(\q,\p)\st \p\q=1\}.$$ 
Clearly, the tangent space to $\Qd$ at a point $(\A,\b)$, translated to the origin, consists of vectors $(\dot \A, \dot\b)\in\Rtt$ satisfying
\be\label{eq:tq}
\dot\b\A+\b\dot\A=0
\ee
\begin{definition}\label{def:hor} Define a rank 2 distribution  $\Dd\subset T\Qd$ as follows:  the elements at  a point $(\A, \b)\in\Qd\subset\Rtt$,  translated to the origin in $\Rtt$, are the vectors $(\dot\A, \dot\b)\in \Rtt$  satisfying, in addition to Equation \eqref{eq:tq}, 
\be\label{eq:dist}
\dot\p =\q\times \dot \q.
\ee
\end{definition}

\begin{prop}\label{prop:hor}
$(\Qd,\Dd)$ is a  flat \ttf-distribution. 
\end{prop}

This appeared in \cite{BHN}. We recall the flatness argument: $(\Qd, \Dd)$ admits $\SL_3(\R)$ as an `obvious' symmetry group, the  restriction to $\Qd$ of the diagonal action on $\Rtt$, $(\A,\b)\mapsto (g\A,\b g^{-1}).$ This is an 8-dimensional group, hence by Cartan's submaximality result, the full symmetry group is in fact 14-dimensional. 

In \cite{BHN} is  given also  an  explicit local  action of $\G$ on $\Qd$. We recall this in Appendix \ref{app:emb}, by  embedding  $(\Qd, \Dd)$ in another model of the \CED, a homogeneous space of $\G$, constructed using split octonions, see Section \ref{sec:diag}.


Here is an  easy to verify extrinsic property of $\Dd$, not found in \cite{BHN}.

\begin{prop}\label{prop:hp}
For any  given point in $\Qd$ and a $\Dd$-direction at the point, the affine line   in $\Rtt$  passing through the given point and tangent to the given direction is contained in $\Qd$ and everywhere tangent to $\Dd$. In fact, these affine  lines are  exactly the rigid curves of $\Dd$ in the sense of \cite{BrH}. 
\end{prop}

Now we come to the second main definition of this article (the first one was  Definition \ref{def:dp}). 

\begin{definition}[Horizontal polygons] \label{def:hp}A  {\em horizontal polygon} in $\Qd$  is a polygon in $\Rtt$ whose  edges  are $\Dd$-horizontal  lines in $\Qd$, as in Proposition \ref{prop:hp}. The polygon is {\em non-degenerate} if every 3 consecutive vertices are non-collinear. 
 \end{definition}

Our next  theorem gives a bijection between the objects defined in Definitions \ref{def:dp} and \ref{def:hp}. 

\begin{theorem}\label{thm:horpol}For every  non-degenerate horizontal $n$-gon in $\Qd$, closed or open,   with  vertices   $(\q_1, \p_1), (\q_2, \p_2),\ldots, (\q_n, \p_n)$,  $n\geq 2$,   the projected pair of polygons,  the first with vertices 
$A_{i}=[\q_i]\in \RPt$ and the second with edges $b_i=[\p_i]\in \RPts$, $i=1,\ldots, n$, is a non-degenerate  dancing pair. Conversely, every non-degenerate dancing pair of   polygons   lifts uniquely to a   non-degenerate  horizontal polygon in $\Qd$.
\end{theorem}

This is proved in  Section \ref{sec:horpol}. 

\newcommand{\Ss}{S^2_{\tt sta}}
\newcommand{\Sm}{S^2_{\tt mov}}

\subsection{Second model: rolling balls}\label{ss:rol}
The configuration space for rolling balls is $\Qr:=S^2\times \SO_3$. A point $(\v,g)\in \Qr$ 
represents a  placement of the moving ball  so that it touches the stationary sphere 
at the point $\rs \v$ and is rotated with respect to some fixed reference orientation by $g\in\SO_3$. Let $\varphi_{(\v,g)}:\Rt\to\Rt$ be the rigid motion $\x\mapsto g\x+(\rs+\rv)\v$. Then the placement of the moving ball 
is given by its image  under $\varphi_{(\v,g)}$. See Figure \ref{fig:roll}. 

\begin{figure}[h]
\centering
\def\svgwidth{.6\textwidth}
\begingroup%
  \makeatletter%
  \providecommand\color[2][]{%
    \errmessage{(Inkscape) Color is used for the text in Inkscape, but the package 'color.sty' is not loaded}%
    \renewcommand\color[2][]{}%
  }%
  \providecommand\transparent[1]{%
    \errmessage{(Inkscape) Transparency is used (non-zero) for the text in Inkscape, but the package 'transparent.sty' is not loaded}%
    \renewcommand\transparent[1]{}%
  }%
  \providecommand\rotatebox[2]{#2}%
  \newcommand*\fsize{\dimexpr\f@size pt\relax}%
  \newcommand*\lineheight[1]{\fontsize{\fsize}{#1\fsize}\selectfont}%
  \ifx\svgwidth\undefined%
    \setlength{\unitlength}{227.08049011bp}%
    \ifx\svgscale\undefined%
      \relax%
    \else%
      \setlength{\unitlength}{\unitlength * \real{\svgscale}}%
    \fi%
  \else%
    \setlength{\unitlength}{\svgwidth}%
  \fi%
  \global\let\svgwidth\undefined%
  \global\let\svgscale\undefined%
  \makeatother%
  \begin{picture}(1,0.73415157)%
    \lineheight{1}%
    \setlength\tabcolsep{0pt}%
    \put(0,0){\includegraphics[width=\unitlength,page=1]{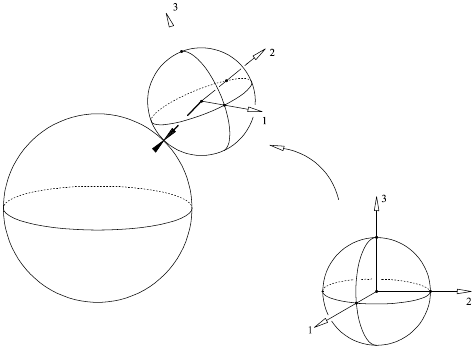}}%
    \put(0.23959089,0.38178586){\color[rgb]{0,0,0}\makebox(0,0)[lt]{\lineheight{1.25}\smash{\begin{tabular}[t]{l}\s$r\v$\end{tabular}}}}%
    \put(0.38365161,0.44496577){\color[rgb]{0,0,0}\makebox(0,0)[lt]{\lineheight{1.25}\smash{\begin{tabular}[t]{l}\s$-r'\v$\end{tabular}}}}%
    \put(0.67843561,0.42143121){\color[rgb]{0,0,0}\makebox(0,0)[lt]{\lineheight{1.25}\smash{\begin{tabular}[t]{l}\s$\varphi_{(\v, g)}$\end{tabular}}}}%
    \put(0,0){\includegraphics[width=\unitlength,page=2]{rolling.pdf}}%
  \end{picture}%
\endgroup%

\caption{The rolling configuration space $S^2\times\SO_3$.}\label{fig:roll}
\end{figure}


For a given element $g\in \SO_3$ and tangent vector $\dot g\in T_g\SO_3$ there is a unique  $\bo\in \Rt$, called the  {\em angular velocity} vector, such that $\dot gg^{-1}\in\so_3$ is given by 
$\v\mapsto\bo\times\v.$ 

\begin{definition}\label{def:rol}
Define $\Dr\subset T\Qr$ as follows. The elements of $\Dr$ at a  point $(\v, g)\in\Qr$ are the   tangent  vectors $(\dot\v, \dot g)\in T_{(\v,g)}\Qr$    satisfying 
\begin{enumerate}[(1)]
\item  $(\rho + 1)\dot\v=\bo \times  \v$,
\item $\bo\cdot\v = 0$.
\end{enumerate}
\end{definition}
See Proposition 1 of \cite{BM} for a derivation of conditions (1) and (2) from `no-slip' and `no-twist' (respectively).

\begin{prop}\hfill
\begin{enumerate}[{\rm (a)}]
\item $\Dr$ is a 2-distribution on $\Qr$, integrable for $\rho=1$ and \ttf\ for $\rho\neq  1$. 
\item For $\rho\neq 1$, the rigid curves of $\Dr$ correspond to rolling along great circles of the stationary sphere. Thus rolling along polygonal curves on the stationary sphere corresponds to piecewise rigid curves of $\Dr$. 
\item $\Dr$ is a flat \ttf-distribution  if and only if $\rho=3$ or $1/3$. 
\end{enumerate} 
\end{prop}
Parts (a) and (b) appeared in \S4.4 of \cite{BrH}. We learned about part (c)   from a  conversation with Robert Bryant. See \cite{BM, BH} for a proof of (c) as well as  attempts to explain the mysterious 3:1 radius ratio.

\subsection{Third model: split octonions}\label{ss:oct}
We use here the  notation of \cite[\S3]{BHN}. See also \cite{J,S} for more details. 

\paragraph{Split octonions.} This is  an $8$-dimensional non-commutative and non-associative real algebra $\tO$. Following Zorn \cite[page 144]{Z}, its  elements can be written  as `vector matrices'
$$\left(\begin{matrix}x&\q \\ \p&y\end{matrix}\right),\ \  x,y\in\R,\ \  \q\in\R^3, \ \  \p\in\Rts,$$
with `vector-matrix-multiplication', 
$$\left(\begin{matrix}x&\q \\ \p&y\end{matrix}\right)\zm\left(\begin{matrix}x'&\q' \\ \p'&y'\end{matrix}\right):=
\left(\begin{array}{lr}xx'-\p'\q& x\q'+ y'\q+ \p\times\p'\\ x'\p+ y\p'+ \q\times\q'&yy'-\p\q'\end{array}\right),$$
where, as before,  we use the `vector products' of formulas \eqref{eq:prod}. 

Conjugation in $\tO$ is given by 
$$\m=\left(\begin{matrix}x&\q \\ \p&y\end{matrix}\right)\mapsto \overline{\m}=\left(\begin{array}{rr}y&-\q \\ -\p&x\end{array}\right),$$
satisfying  
$$\overline{\overline{\m}}=\m, \ \  \overline{\m\zm\m'}=\overline{\m'}\zm\overline{\m},
\ \  \m\zm\overline{\m}=\<\m,\m\>\bf 1,$$
where $\bf 1=\left(\begin{smallmatrix}1&0 \\ 0&1\end{smallmatrix}\right)$ is the mutiplicative unit and
\be\label{eq:ip} \<\m,\m\>:= xy+\p\q
\ee
is a quadratic form  of signature $(4,4)$ on $\tO$. 

Define as usual $$\Re(\m)=(\m+\overline\m)/2, \ \  \Im(\m)=(\m-\overline\m)/2, $$
so that $$\tO=\Re(\tO)\oplus\Im(\tO),$$ where 
$\Re(\tO)=\R\bf 1$ and $\Im(\tO)$ are `traceless' vector matrices of the form $\left(\begin{smallmatrix}x&\q \\ \p&-x\end{smallmatrix}\right)$.

Let $\RP^6=\P(\Im(\tO))=\left(\Im(\tO)\setminus 0\right) / \R^*$ (projectivized imaginary split octonions) and $\Qo\subset \RP^6$ the projectivized null cone; that is, 
$$\Qo=\{[\m]\st \m\in\Im(\tO),\ \m\neq 0, \ \<\m, \m\>=0\}.$$

\begin{definition}$\Do\subset T\Qo$ is the distribution whose elements at a point $[\m]\in\Qo$ are given by the {\em annihilator} 
 $$\m^0=\{\m'\in \Im(\tO)\st \m\zm\m'=0\},$$
as follows.  For a non-zero null $\m\in\Im(\tO)$, $\m^0$ is a 3-dimensional subspace of $\Im(\tO)$, containing $\R\m$ and tangent to the null cone at $\m$, so descends to a projective 2-plane, tangent  to $\Qo$ at $[\m]$. Its tangent space at $[\m]$ is the fiber of $\Do$ at $[\m].$ 
\end{definition}

We have a projective analog of  Proposition \ref{prop:hp} and Definition \ref{def:hp}.

\begin{prop}\label{prop:hpp}
For any  given point in $\Qo$ and a $\Do$-direction at the point, the projective line   in $\RP^6$  passing through the given point and tangent to the given direction is contained in $\Qo$ and everywhere tangent to $\Do$. These are   the rigid curves of $\Do$. 
\end{prop}

\begin{definition}[Horizontal polygons] \label{def:hpp}A  {\em horizontal polygon} in $\Qo$  is a polygon in $\RP^6$ whose  edges  are  $\Do$-horizontal  projective lines in $\Qo$, as in Proposition \ref{prop:hpp}. The polygon is {\em non-degenerate} if every 3 consecutive vertices are non-collinear. 
 \end{definition}

\paragraph{$\G$-symmetry.} The automorphism group of $\tO$, i.e., the subgroup of  $\GL(\tO)$ preserving octonion multiplication,  is a non-compact 14-dimensional connected simple Lie group, denoted by $\G$.  See Appendix \ref{app:sym}.
The $\G$-action on  $\tO$ preserves the splitting $\tO=\Re(\tO)\oplus\Im(\tO)$ and acts trivially on $\Re(\tO)$. The action on $\Im(\tO)$ induces an effective action on $\P(\Im(\tO))$, clearly preserving $(\Qo,\Do)$. 
Therefore, one has

 \begin{prop}
 $(\Qo,\Do)$ is a flat \ttf-distribution, whose  symmetry  group is  $\G$. 
 \end{prop}

The  Lie algebra of $\G$ is $\g_2\subset \so_{4,3}\subset\so_{4,4}$; the first  inclusion is due to the $\G$-invariance of the  inner product of equation \eqref{eq:ip}, restricted to $\Im(\tO)$, the second from the inclusion $\Im(\tO)\subset\tO$.

\subsection{Putting it all together}
To prove  Theorem \ref{thm:main}, we combine the three models of the \CED\ presented so far, as summed up in the following diagram. 

\be\label{eq:diag}
\begin{tikzcd}
&\tQo \arrow[d]&\tQr\arrow[d]\arrow[l, "\Phi"', "\widetilde{}"]\\
\Qd \arrow[r,"\iota", hookrightarrow]&\Qo&\Qr
\end{tikzcd}
\ee

\mn

In this diagram appear the underlying manifolds of  our three models for the \CED, $\Qd, \Qo$ and $\Qr$, the universal covers of the last two (both of which are diffeomorphic to $S^2\times S^3$), as well as the following maps:
%
\begin{itemize}
\item An embedding $\iota:\Qd\hookrightarrow \Qo$. One defines first an affine chart $\Rtt\hookrightarrow\P(\Im(\tO))=\RP^6$, 
$$ (\q,\p)\mapsto [\m],\ \mbox{where } \m=\left(\begin{matrix}1&\q \\ \p&-1\end{matrix}\right),$$ 
then restrict   to $\Qd=\{\p\q=1\}\subset\Rtt.$ The image $\iota(\Qd)\subset\Qo$ is the complement in $\Qo$ of the `hyperplane section'  
$\{\left[\left(\begin{smallmatrix}x&\q \\ \p&-x\end{smallmatrix}\right)\right]\st \ x=\b\A=0\}$. 
\item  The double  covers  $\tQo\to\Qo$, $\tQr\to\Qr$ (the vertical arrows).
\begin{itemize}
\item Recall that $\Qo$ is the projectivized null cone in $\RP^6=\left(\Im(\tO)\setminus 0\right)/\R^*.$ If we quotient  by $\R^+$  instead of $\R^*$ we get the `spherized' null cone $\tQo\subset \left(\Im(\tO)\setminus 0\right)/\R^+=S^6,$ and a double cover $\tQo\to\Qo$, $\R^+\m\mapsto \R^*\m.$ We use this double cover to pull back $\Do$ to $\tDo$ on $\tQo$. 

\item Recall that  $\Qr=S^2\times\SO_3$,  define $\tQr:=S^2\times S^3$ and use the usual double cover $ S^3\to \SO_3$ to define the double cover $\tQr\to\Qr$. Explicitly,  to  an element $q\in S^3\subset \R^4=\H$ one associates the orthogonal transformation of $\R^3=\Im(\H)$, $\v\mapsto q\v \bar q$. Again, use this double cover to pull back $\Dr$ to $\tDr$ on $\tQr$. 
\end{itemize}
\item A  diffeomorphism\footnote{This diffeomorphism is essentially the same as the one constructed in \cite[Proposition 2]{BH}.}
   $\Phi:\tQr\to\tQo.$  For $(\v,q)\in \tQr=S^2\times S^3\subset \Im(\H)\times \H$ one defines\footnote{
   Note that for  this formula to make sense we identify $\R^3$ with $\Rts$ using the standard Euclidean inner product.}
\be\label{eqn:Phi}
\Phi (\v , q )=\R^+\begin{pmatrix} \mbox{Re}(\v q) & \v +\mbox{Im}(\v q)\\ \v-\mbox{Im}(\v q) & -\mbox{Re}(\v q)\end{pmatrix}.
\ee
\end{itemize}

\begin{theorem}\label{thm:diag}In diagram \eqref{eq:diag}, $\Phi$ is a diffeomorphism and all maps preserve the Cartan-Engel  distributions on the respective spaces. 
\end{theorem}

This is proved in Section \ref{sec:diag}. 
Now we can give a more precise statement of our  main theorem (Theorem \ref{thm:main}). Let $\Qo_*=\iota(\Qd)\subset \Qo$ (the complement of the hyperplane section $x=0$), 
and similarly $\tQo_*\subset \tQo, \tQr_*\subset\tQr, \Qr_*\subset\Qr$ the corresponding submanifolds under the maps of diagram \eqref{eq:diag}. A polygonal horizontal path in  
$\Qo$ (respectively $\tQo, \tQr, \Qr$)  is {\em generic} if all its vertices lie in $\Qo_*$ (respectively $\tQo_*, \tQr_*, \Qr_*$). 
A pair $(\Gamma, q)$, where $\Gamma$ is a closed non-degenerate spherical $n$-gon with vertices $\v_1,\ldots, \v_n$ and 
$q\in S^3$, is {\em generic} if the horizontal lift of $\Gamma$ to $\tQr=S^2\times S^3$, 
starting 
at $(\v_1,
q)$,
is a generic horizontal polygon. 
 
\mn {\bf Corollary} (Theorem \ref{thm:main}). \label{cor:corresp}
{\em Diagram \eqref{eq:diag} defines a bijective  correspondence between non-degenerate  dancing pairs of closed $n$-gons in $\RP^2$ and  generic pairs $([\Gamma], q)$, where  $q\in S^3$ and $[\Gamma]$ is an equivalence class of   non-degenerate closed spherical $n$-gon  with trivial lifted  rolling monodromy for $\rho=3$. }\\

In section \ref{ss:pfmain} we shall prove Theorem \ref{thm:main} using Theorem \ref{thm:diag}.

\section{Proofs}

\subsection{Theorem  \ref{thm:horpol}}\label{sec:horpol}

Using  induction on $n\geq 2$, the proof  reduces to the following two lemmas,  the $n=2,3$ cases of the theorem.

\mn\begin{minipage}{.7\textwidth}
\begin{lemma}[Theorem  \ref{thm:horpol} for $n=2$]\label{lemma:2gons}
Consider a horizontal non-degenerate 2-gon in  $\Qd$, given by  a pair of vertices $q_1, q_2\in\Qd$ such that the line $q_1q_2$ is horizontal. Then the projected pair of 2-gons, with vertices  $(A_1, A_2)$ and edges $(b_1,b_2)$, is inscribed, i.e., $b_1b_2\in A_1A_2.$ Conversely,    every inscribed pair of  2-gons   lifts  to  a unique  horizontal 2-gon in $\Qd$. 
\end{lemma}
\end{minipage}
\begin{minipage}{.3\textwidth}
\centering
\def\svgwidth{.85\textwidth}
\begingroup%
  \makeatletter%
  \providecommand\color[2][]{%
    \errmessage{(Inkscape) Color is used for the text in Inkscape, but the package 'color.sty' is not loaded}%
    \renewcommand\color[2][]{}%
  }%
  \providecommand\transparent[1]{%
    \errmessage{(Inkscape) Transparency is used (non-zero) for the text in Inkscape, but the package 'transparent.sty' is not loaded}%
    \renewcommand\transparent[1]{}%
  }%
  \providecommand\rotatebox[2]{#2}%
  \newcommand*\fsize{\dimexpr\f@size pt\relax}%
  \newcommand*\lineheight[1]{\fontsize{\fsize}{#1\fsize}\selectfont}%
  \ifx\svgwidth\undefined%
    \setlength{\unitlength}{180.32876803bp}%
    \ifx\svgscale\undefined%
      \relax%
    \else%
      \setlength{\unitlength}{\unitlength * \real{\svgscale}}%
    \fi%
  \else%
    \setlength{\unitlength}{\svgwidth}%
  \fi%
  \global\let\svgwidth\undefined%
  \global\let\svgscale\undefined%
  \makeatother%
  \begin{picture}(1,0.51928205)%
    \lineheight{1}%
    \setlength\tabcolsep{0pt}%
    \put(0,0){\includegraphics[width=\unitlength,page=1]{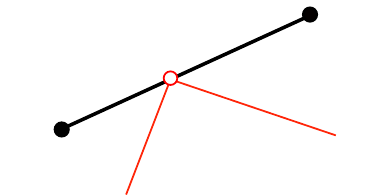}}%
    \put(0.38616494,0.04098082){\color[rgb]{0,0,0}\makebox(0,0)[lt]{\lineheight{0}\smash{\begin{tabular}[t]{l}$b_1$\end{tabular}}}}%
    \put(-0.00300521,0.19843143){\color[rgb]{0,0,0}\makebox(0,0)[lt]{\lineheight{0}\smash{\begin{tabular}[t]{l}$A_1$\end{tabular}}}}%
    \put(-0.13736302,0.11515343){\color[rgb]{0,0,0}\makebox(0,0)[lt]{\lineheight{0}\smash{\begin{tabular}[t]{l} \end{tabular}}}}%
    \put(0.75668293,0.23589039){\color[rgb]{0,0,0}\makebox(0,0)[lt]{\lineheight{0}\smash{\begin{tabular}[t]{l}$b_2$\end{tabular}}}}%
    \put(0.87451802,0.48642175){\color[rgb]{0,0,0}\makebox(0,0)[lt]{\lineheight{0}\smash{\begin{tabular}[t]{l}$A_2$\end{tabular}}}}%
  \end{picture}%
\endgroup%

\end{minipage}

\begin{proof}Suppose $q_i=(\A_i, \b_i)$, $i=1,2$, such that  $q_1q_2$  is horizontal. Then, by definition (Proposition \ref{prop:hor}), $\q_1\times\q_2=\p_2-\p_1$. This equation  implies $(\q_1\times\q_2)(\p_1\times\p_2)=0$, i.e., $b_1b_2\in A_1A_2.$ 

Conversely, suppose   an inscribed pair of  2-gons is given, i.e., $b_1b_2\in A_1A_2$, with $b_1\neq b_2, A_1\neq A_2$, $A_i\not\in b_i,$ $i=1,2.$  The horizontality conditions on a lift $(\q_i,\p_i)\in \Qd$, $i=1,2$, is $\q_1\times\q_2=\p_2-\p_1$ (see Proposition \ref{prop:hor}).

Picking an arbitrary lift   $(\q_i,\p_i)\in \Qd$, any other lift is of the form  $(x_i\q_i ,  \p_i/x_i),$ for some $x_i\in\R\setminus 0$,  $i=1,2,$  and the  horizontality condition  is $(x_1\q_1)\times(x_2\q_2)=\p_2/x_2-\p_1/x_1$.

 Now  $b_1b_2\in A_1A_2$  implies  
$\q_1\times\q_2 =\lambda_1\p_1+\lambda_2\p_2$ for some  $\lambda_i\neq 0,$ 
therefore $x_1x_2 (\lambda_1\p_1 +\lambda_2\p_2)= \p_2/x_2-\p_1/x_1$. Thus the horizontality condition becomes
 $(x_1)^2x_2\lambda_1=-1$ and $x_1(x_2)^2\lambda_2=1$.
This system has a unique solution $x_1=\sqrt[3]{\lambda_2/(\lambda_1)^2}$, $x_2=-\sqrt[3]{\lambda_1/(\lambda_2)^2}$, as needed. 
\end{proof}

\noindent\begin{minipage}{.6\textwidth}
\begin{lemma}[Theorem  \ref{thm:horpol} for $n=3$]\label{lemma:3gons}
Consider a  horizontal non-degenerate 3-gon in $\Qd$, given by  3 vertices $q_1, q_2, q_3\in\Qd$ such that $q_1q_2$ and $q_2q_3$ are horizontal lines. Then the projected pair of inscribed 3-gons with vertices $A_1,A_2,A_3$ and edges $b_1, b_2, b_3$ is a dancing pair; i.e., it satisfies 
\begin{equation}\label{eq:dp1}
[A_{2}, B_{1}, A_{1},D]+[B_{2}, A_{2},C, A_{3}]=0, 
\end{equation}
where $B_1=b_1b_2, B_2=b_2b_3, C:=b_1 a_{2}$ and $ D:=a_1b_{2}$. 
Conversely, a dancing pair of   3-gons  lifts uniquely to a horizontal 3-gon in $Q$. 
\end{lemma}
\end{minipage}
\quad
\begin{minipage}{.4\textwidth}
\centering\def\svgwidth{.9\textwidth}
\begingroup%
  \makeatletter%
  \providecommand\color[2][]{%
    \errmessage{(Inkscape) Color is used for the text in Inkscape, but the package 'color.sty' is not loaded}%
    \renewcommand\color[2][]{}%
  }%
  \providecommand\transparent[1]{%
    \errmessage{(Inkscape) Transparency is used (non-zero) for the text in Inkscape, but the package 'transparent.sty' is not loaded}%
    \renewcommand\transparent[1]{}%
  }%
  \providecommand\rotatebox[2]{#2}%
  \newcommand*\fsize{\dimexpr\f@size pt\relax}%
  \newcommand*\lineheight[1]{\fontsize{\fsize}{#1\fsize}\selectfont}%
  \ifx\svgwidth\undefined%
    \setlength{\unitlength}{255.7929476bp}%
    \ifx\svgscale\undefined%
      \relax%
    \else%
      \setlength{\unitlength}{\unitlength * \real{\svgscale}}%
    \fi%
  \else%
    \setlength{\unitlength}{\svgwidth}%
  \fi%
  \global\let\svgwidth\undefined%
  \global\let\svgscale\undefined%
  \makeatother%
  \begin{picture}(1,0.79034222)%
    \lineheight{1}%
    \setlength\tabcolsep{0pt}%
    \put(0,0){\includegraphics[width=\unitlength,page=1]{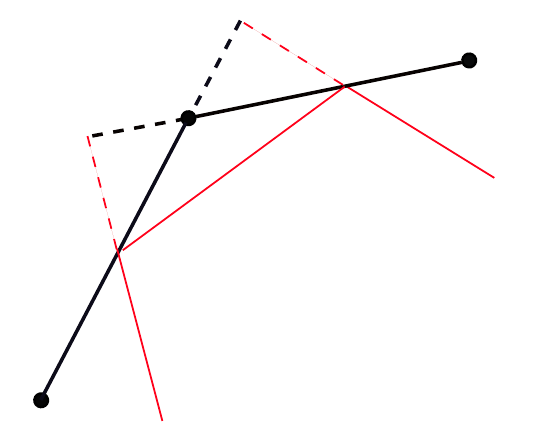}}%
    \put(0.31349067,0.07602948){\color[rgb]{0,0,0}\makebox(0,0)[lt]{\lineheight{0}\smash{\begin{tabular}[t]{l}\s$b_1$\end{tabular}}}}%
    \put(-0.00095679,0.08266454){\color[rgb]{0,0,0}\makebox(0,0)[lt]{\lineheight{0}\smash{\begin{tabular}[t]{l}\s$A_1$\end{tabular}}}}%
    \put(0.27069983,0.61120929){\color[rgb]{0,0,0}\makebox(0,0)[lt]{\lineheight{0}\smash{\begin{tabular}[t]{l}\s$A_2$\end{tabular}}}}%
    \put(0.08812801,0.31708228){\color[rgb]{0,0,0}\makebox(0,0)[lt]{\lineheight{0}\smash{\begin{tabular}[t]{l}\s$B_1$\end{tabular}}}}%
    \put(0.63682035,0.67222607){\color[rgb]{0,0,0}\makebox(0,0)[lt]{\lineheight{0}\smash{\begin{tabular}[t]{l}\s$B_2$\end{tabular}}}}%
    \put(0.84043852,0.71480873){\color[rgb]{0,0,0}\makebox(0,0)[lt]{\lineheight{0}\smash{\begin{tabular}[t]{l}\s$A_3$\end{tabular}}}}%
    \put(0.88601888,0.51021666){\color[rgb]{0,0,0}\makebox(0,0)[lt]{\lineheight{0}\smash{\begin{tabular}[t]{l}\s$b_3$\end{tabular}}}}%
    \put(0.43600523,0.40108886){\color[rgb]{0,0,0}\makebox(0,0)[lt]{\lineheight{0}\smash{\begin{tabular}[t]{l}\s$b_2$\end{tabular}}}}%
    \put(0,0){\includegraphics[width=\unitlength,page=2]{3gons.pdf}}%
    \put(0.06949166,0.56970857){\color[rgb]{0,0,0}\makebox(0,0)[lt]{\lineheight{0}\smash{\begin{tabular}[t]{l}\s$C$\end{tabular}}}}%
    \put(0.47784504,0.76717639){\color[rgb]{0,0,0}\makebox(0,0)[lt]{\lineheight{0}\smash{\begin{tabular}[t]{l}\s$D$\end{tabular}}}}%
    \put(0,0){\includegraphics[width=\unitlength,page=3]{3gons.pdf}}%
    \put(0.06803045,0.18996344){\color[rgb]{0,0,0}\makebox(0,0)[lt]{\lineheight{0}\smash{\begin{tabular}[t]{l}\s$a_1$\end{tabular}}}}%
    \put(0.47546903,0.62922673){\color[rgb]{0,0,0}\makebox(0,0)[lt]{\lineheight{0}\smash{\begin{tabular}[t]{l}\s$a_2$\end{tabular}}}}%
  \end{picture}%
\endgroup%

\end{minipage}

\begin{proof} Let $q_i=(\q_i, \p_i), i=1,2,3.$ By Lemma \ref{lemma:2gons}, the projected pair of 3-gons is inscribed and we need to show that Equation \eqref{eq:dp1} is satisfied. We give homogeneous coordinates to all points involved:
$${\bf B}_1:=\p_1\times\p_2,\
{\bf B}_2:=\p_2\times\p_3,\
{\bf C}:=\p_1\times (\q_2\times\q_3),\
{\bf D}:=\p_3\times (\q_1\times\q_2).
$$

We now  write these  expressions in terms of $\q_1,\q_2, 
\q_3$. From the horizontality condition  $\p_2-\p_1=\q_1\times\q_2$ and the vector identity \be\label{eq:vi}
{\bf a}\times 
({\bf B}\times {\bf C})=({\bf a}{\bf C}){\bf B}-({\bf a}{\bf B}){\bf C}
\ee
follows
$$\p_1\times\p_2=\p_1\times(\p_1+\q_1\times\q_2)={(\p_{1}\q_{2})}\q_1 -\q_2.$$ 
Now $(\q_2-\q_1, \p_2-\p_1)$ is tangent to $\Qd=\{\p\q=1\}$ at $(\q_1, \p_1)$, hence 
\begin{align*}
0&=(\p_2-\p_1)\q_1+\p_1(\q_2-\q_1)=
\p_2\q_1+\p_1\q_2-2=\\
&=(\p_1-\q_1\times\q_2)\q_1+\p_1\q_2-2=\p_1\q_2-1,
\end{align*}
hence ${\bf B}_1=\q_1-\q_2.$ Similarly, ${\bf B}_2=\q_2-\q_3.$ Using again equation  \eqref{eq:vi}, ${\bf C}=(\p_1\q_3)\q_2-\q_3$ and ${\bf D}=\q_1-(\p_3\q_1)\q_2.$

Now it is easy to show that if 4 collinear points $A_1, \dots , A_4\in\RP^2$ are given by
homogeneous coordinates $\q_i\in\R^3$, such that $\q_3 = \q_1 + \q_2, \q_4 = k\q_1 + \q_2,$
then $[A_1, A_2, A_3, A_4] = k$ (see for example \cite[Theorem 74, page 105]{K}). Using this formula and the above
expressions for ${\bf B}_1, {\bf B}_2, {\bf C}, {\bf D},$ 
we get $$ [A_2,B_1,A_1, D]=1-\p_3\q_1,\quad [B_2,A_2, C, A_3]=1-\p_1\q_3,
$$ 
hence $$
 [A_2,B_1,A_1, D]+ [B_2,A_2, C, A_3] =2-\p_3\q_1-\p_1\q_3.
 $$
 
 Next 
 \begin{align*}
 \p_3\q_1+\p_1\q_3  &= (\p_2 +\q_2\times\q_3)\q_1 +
 (\p_2-\q_1\times\q_2)\q_3 \\
 &={\p_2\q_1 }+{\p_2\q_3}=2,
 \end{align*}
  so Equation \eqref{eq:dp1} is satisfied, as needed.

\medskip In the other direction,  suppose an inscribed pair of 3-gons is given,  with vertices $A_1, A_2, A_3$ and edges $b_1,b_2,b_3$, satisfying Equation \eqref{eq:dp1}. By  Lemma \ref{lemma:2gons}, we can uniquely lift $(A_1,b_1), (A_2,b_2)$ to points $(\q_1,\p_1), (\q_2,\p_2)\in\Qd$ on a horizontal line. 
  Likewise, we can uniquely  lift $(A_2,b_2), (A_3,b_3)$ to points $(\lambda\q_2, \p_2/\lambda), (\q_3,\p_3)\in\Qd$  on a horizontal line, for some $\lambda\neq 0.$ See Figure \ref{fig:match}. 

\begin{figure}[H]
\centering
\def\svgwidth{.6\textwidth}
\begingroup%
  \makeatletter%
  \providecommand\color[2][]{%
    \errmessage{(Inkscape) Color is used for the text in Inkscape, but the package 'color.sty' is not loaded}%
    \renewcommand\color[2][]{}%
  }%
  \providecommand\transparent[1]{%
    \errmessage{(Inkscape) Transparency is used (non-zero) for the text in Inkscape, but the package 'transparent.sty' is not loaded}%
    \renewcommand\transparent[1]{}%
  }%
  \providecommand\rotatebox[2]{#2}%
  \newcommand*\fsize{\dimexpr\f@size pt\relax}%
  \newcommand*\lineheight[1]{\fontsize{\fsize}{#1\fsize}\selectfont}%
  \ifx\svgwidth\undefined%
    \setlength{\unitlength}{386.20503103bp}%
    \ifx\svgscale\undefined%
      \relax%
    \else%
      \setlength{\unitlength}{\unitlength * \real{\svgscale}}%
    \fi%
  \else%
    \setlength{\unitlength}{\svgwidth}%
  \fi%
  \global\let\svgwidth\undefined%
  \global\let\svgscale\undefined%
  \makeatother%
  \begin{picture}(1,0.24236462)%
    \lineheight{1}%
    \setlength\tabcolsep{0pt}%
    \put(-0.00063371,0.03384359){\color[rgb]{0,0,0}\makebox(0,0)[lt]{\lineheight{0}\smash{\begin{tabular}[t]{l}\s$(\A_1,\b_1)$\end{tabular}}}}%
    \put(0.50428604,0.02707826){\color[rgb]{0,0,0}\makebox(0,0)[lt]{\lineheight{0}\smash{\begin{tabular}[t]{l}\s$(\A_2,\b_2)$\end{tabular}}}}%
    \put(0.28455756,0.22702133){\color[rgb]{0,0,0}\makebox(0,0)[lt]{\lineheight{0}\smash{\begin{tabular}[t]{l}\s$(\lambda\A_2,\b_2/\lambda)$\end{tabular}}}}%
    \put(0.84783821,0.06172052){\color[rgb]{0,0,0}\makebox(0,0)[lt]{\lineheight{0}\smash{\begin{tabular}[t]{l}\s$(\A_3,\b_3)$\end{tabular}}}}%
    \put(0,0){\includegraphics[width=\unitlength,page=1]{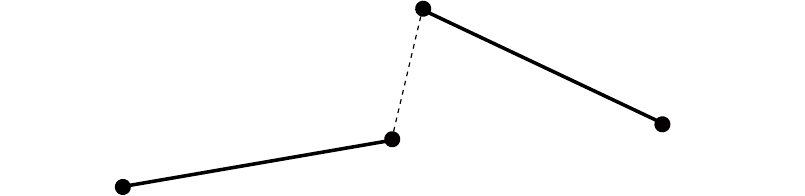}}%
  \end{picture}%
\endgroup%

\captionof{figure}{Lemma \ref{lemma:3gons}. }
 \label{fig:match}
\end{figure}

We now show that the dancing condition \eqref{eq:dp1} implies $\lambda=1$, i.e., the three lifted points $(\A_1,\b_1),(\A_2,\b_2),(\A_3,\b_3)$  are the vertices of a horizontal 3-gon.

We use the following:
\begin{itemize}
\item $\p_2=\p_1+\q_1\times\q_2$, $\p_3={1\over\lambda}\p_2 + \lambda\q_2\times\q_3$;
\item $A_1=[\q_1]$, $A_2=[\q_2]$
\item $B_1=[\p_1\times\p_2]=[\q_1-\q_2]$, $B_2=[\p_2\times\p_3]=[\lambda\q_2-\q_3]$;
\item $C=[\p_1\times(\q_2\times\q_3)]=[(\p_1\q_3)\q_2-\q_3]$;
\item $D=[\p_3\times(\q_\times\q_2)]=[(\p_3\q_2)\q_1-(\p_3\q_1)\q_2]=[{1\over\lambda}\q_1 - (\p_3\q_1)\q_2]$.
\end{itemize}
A similar computation as above gives
\begin{equation*}
\begin{split}
[A_2,B_1,A_1,D] &= 1-\lambda\p_3\q_1\\
[B_2,A_2, C, A_3] &= 1-{1\over\lambda}\p_1\q_3
\end{split}
\end{equation*}
Now, from $\p_3={1\over\lambda}\p_2+\lambda\q_2\times\q_3$, we get that
$$
\p_3\q_1={1\over\lambda}+ \lambda (\q_1\times\q_2)\q_3;
$$whereas, from $\p_2=\p_1+\q_1\times\q_2$ we obtain
$$
\p_1\q_3=\lambda -(\q_1\times\q_2)\q_3,
$$since $\p_2\q_3=\lambda$.

\medskip

Thus,
\begin{equation*}
\begin{split}
0&= [A_2,B_1,A_1,D]+ [B_2,A_2, C, A_3]\\ 
&= 2-\lambda\p_3\q_1-{1\over\lambda}\p_1\q_3\\
&= 2-(1+\lambda^2 (\q_1\times\q_2)\q_3)-(1-{1\over\lambda}(\q_1\times\q_2)\q_3)\\
&=\left( {1\over\lambda}-\lambda^2\right)(\q_1\times\q_2)\q_3
\end{split}
\end{equation*}
Since $\q_1, \q_2, \q_3$ are non-collinear $(\q_1\times\q_2)\q_3\neq 0$, therefore, $\lambda=1$. 
\end{proof}

Theorem \ref{thm:horpol} now follows from these previous two lemmas by induction on $n$, using  Lemma \ref{lemma:3gons} for the inductive step. 
\qed

\subsection{Theorem \ref{thm:regularpolys}}\label{sec:reg}

We identify $\R^3=\Im(\H)$, $(x,y,z)\mapsto x\i+y\j+z\k$, 
and use repeatedly  the following well known facts: \begin{enumerate}[(i)]
\item If $\u\in\R^3$  is  a unit vector and  $\theta\in\R$ then $\v\mapsto \v'=e^{\theta\u}\v e^{-\theta\u}$ is the rotation about the axis $\R\u$ by the angle $2\theta$, in the sense given by the `right hand rule' ($\det(\v,\v',\u)>0$).   
\item $e^{\theta\u}=\cos\theta+\u \sin\theta.$
\end{enumerate}

\mn\paragraph{Proof of Part (b).} (Part (a) will follow easily from part (b).) Let  $\v_0,...,\v_{n-1}$ be the vertices of $\Gamma$, arranged on a circle of (spherical) radius $\phi$, centered at the north pole $\k$ of the stationary sphere (which we assume here to be a unit sphere, to simplify notation). Let $w$ be the winding number of $\Gamma$ about  $\k$, so that $\theta:=2\pi w/n$ is the angle of rotation at  $\k$ sending $\v_i\mapsto \v_{i+1}$. That is, setting $q:=e^{\theta\k/2}$, one has $\v_i=q^i\v_0\bar q^i.$ As  we roll the moving sphere  along  the  edge of $\Gamma$ joining $\v_i$ to $\v_{i+1}$, an arc of a great circle  of length $\delta$, the moving sphere  rotates about the axis  through its center in the direction of 
$$\u_i:={\v_i\times \v_{i+1}\over \| \v_i\times \v_{i+1}\|}$$ 
by an angle of $4\delta$ (due to the 3:1 radius ratio). See Fig. \ref{fig:lift}. 

\begin{figure}[h]
\centering
\def\svgwidth{.6\textwidth}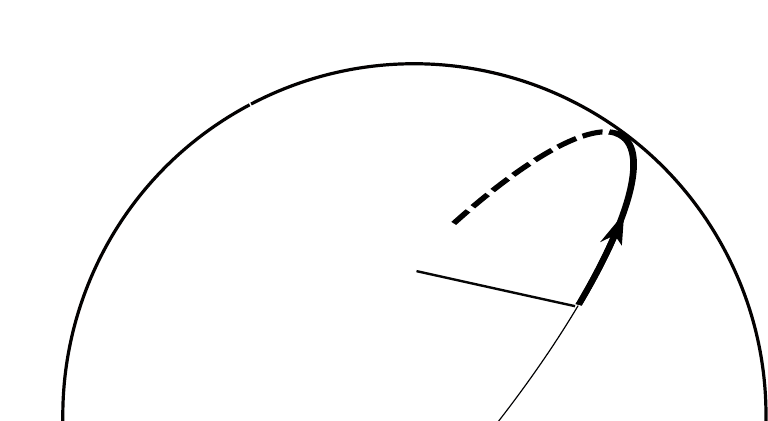
\caption{\ }\label{fig:lift}
\end{figure}

Thus the lifted monodromy due to rolling along this  edge is $g_i:=e^{2\delta \u_i},$ and the total lifted monodromy is 
$$g:=g_{n-1}\cdots g_1 g_0.$$
Now, clearly $\u_i=q^i\u_0\bar q^i,$ hence $g_i=q^i g_0\bar q^i,$ so 
$$g=\left(q^{n-1}g_0\bar q^{n-1}\right)\left(q^{n-2}g_0\bar q^{n-2}\right)\cdots \left(qg_0\bar q \right) g_0=q^n(\bar q g_0)^n.
$$
Next  
$$q^n=e^{n\theta\k/2}=e^{n(2\pi w/n)\k/2}=e^{\pi\k w}=(-1)^w,
$$
hence $g=(-1)^w(\bar q g_0)^n.$ It follows that  the trivial lifted monodromy condition, $g=1$, is equivalent to 
\be (\bar q g_0)^n=(-1)^w.\ee
\begin{lemma}\label{lemma:p}
The condition $p^n=(-1)^w,$ for a unit quaternion $p$, is equivalent to the existence of an integer $w'\equiv w$ (mod 2) such that $\Re(p)=\cos(\pi w'/n).$
\end{lemma}

\begin{proof}Write $p=e^{t\u}=\cos t+\u \sin t, $ for some unit imaginary quaternion $\u$ and $t\in\R$. Then $p^n=\cos (n t)+\u \sin (n t)=(-1)^w \Leftrightarrow \cos(nt)=(-1)^w\Leftrightarrow nt=\pi w'$ for some integer $w'\equiv w$ (mod 2). 
\end{proof}

To apply the last  Lemma to $p=\bar q g_0$ we calculate its real part, 
\begin{align}\label{eq:re}
\begin{split}
\Re(\bar q g_0)&=\Re\left(e^{-\theta \k/2}e^{2\delta \u_0}\right)=\\
&=\cos(\theta/2)\cos (2\delta)+(\k\cdot \u_0) \sin(\theta/2)\sin(2\delta) .
\end{split}
\end{align}
Next we have 
$$\u_0={\v_0\times \v_1\over \|\v_0\times \v_1\|}, \ \ 
\|\v_0\times \v_1\|=\sin\delta,\ \      
 (\v_0\times \v_1)\cdot\k=\sin\theta\sin^2\phi,
 $$
 thus
$$
\k\cdot \u_0={\sin^2\phi\sin\theta\over \sin\delta}.
$$

\noindent\begin{minipage}{.65\textwidth}
\hspace{1.5em}As for $\delta$, we consider the spherical right triangle with vertices $\k, \v_0, M$, where  $M$ is the midpoint of the edge of $\Gamma$ joining  $\v_0$ with $\v_1$. See Figure \ref{fig:trigo}.

\hspace{1.5em} By standard spherical trigonometry (e.g.,  formula (R3) of  \cite{W}), 
$$
 \sin(\delta/2)=\sin\phi  \sin(\theta/2).
$$
\end{minipage}
\quad
\begin{minipage}{.35\textwidth}
\centering
\def\svgwidth{1\textwidth}
\begingroup%
  \makeatletter%
  \providecommand\color[2][]{%
    \errmessage{(Inkscape) Color is used for the text in Inkscape, but the package 'color.sty' is not loaded}%
    \renewcommand\color[2][]{}%
  }%
  \providecommand\transparent[1]{%
    \errmessage{(Inkscape) Transparency is used (non-zero) for the text in Inkscape, but the package 'transparent.sty' is not loaded}%
    \renewcommand\transparent[1]{}%
  }%
  \providecommand\rotatebox[2]{#2}%
  \newcommand*\fsize{\dimexpr\f@size pt\relax}%
  \newcommand*\lineheight[1]{\fontsize{\fsize}{#1\fsize}\selectfont}%
  \ifx\svgwidth\undefined%
    \setlength{\unitlength}{165.58499908bp}%
    \ifx\svgscale\undefined%
      \relax%
    \else%
      \setlength{\unitlength}{\unitlength * \real{\svgscale}}%
    \fi%
  \else%
    \setlength{\unitlength}{\svgwidth}%
  \fi%
  \global\let\svgwidth\undefined%
  \global\let\svgscale\undefined%
  \makeatother%
  \begin{picture}(1,0.90584841)%
    \lineheight{1}%
    \setlength\tabcolsep{0pt}%
    \put(0,0){\includegraphics[width=\unitlength,page=1]{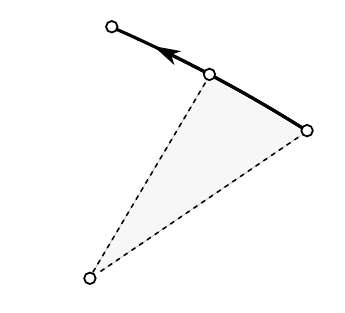}}%
    \put(0.77137979,0.66565637){\color[rgb]{0,0,0}\makebox(0,0)[lt]{\lineheight{1.25}\smash{\begin{tabular}[t]{l}\small${\delta\over 2}$\end{tabular}}}}%
    \put(0.4449804,0.31433784){\color[rgb]{0,0,0}\makebox(0,0)[lt]{\lineheight{1.25}\smash{\begin{tabular}[t]{l}\small${\theta\over 2}$\end{tabular}}}}%
    \put(0.34756282,0.86562524){\color[rgb]{0,0,0}\makebox(0,0)[lt]{\lineheight{1.25}\smash{\begin{tabular}[t]{l}\s$\v_1$\end{tabular}}}}%
    \put(0.2271492,0.00522059){\color[rgb]{0,0,0}\makebox(0,0)[lt]{\lineheight{1.25}\smash{\begin{tabular}[t]{l}\s$\k$\end{tabular}}}}%
    \put(0,0){\includegraphics[width=\unitlength,page=2]{trigo.pdf}}%
    \put(0.91769829,0.54286562){\color[rgb]{0,0,0}\makebox(0,0)[lt]{\lineheight{1.25}\smash{\begin{tabular}[t]{l}\s$\v_0$\end{tabular}}}}%
    \put(0.61795757,0.73301611){\color[rgb]{0,0,0}\makebox(0,0)[lt]{\lineheight{1.25}\smash{\begin{tabular}[t]{l}\s$M$\end{tabular}}}}%
    \put(0.6091557,0.27586322){\color[rgb]{0,0,0}\makebox(0,0)[lt]{\lineheight{1.25}\smash{\begin{tabular}[t]{l} \s$\phi$\end{tabular}}}}%
  \end{picture}%
\endgroup%

\captionof{figure}{\ }\label{fig:trigo}
\end{minipage}


Using the  last two displayed  equations 
in \eqref{eq:re}, we obtain, after some simplification, 
$$\Re(\bar q g_0)=\cos\left(\theta/2\right)
\left[1-4\sin^2\left(\theta/2\right)\sin^2\phi\right].
$$
Recalling that $\theta=2\pi w/n$ and the lifted monodromy triviality condition  $\Re(\bar q g_0)=\cos(\pi w'/n)$ for some $w'\equiv w$ (mod 2), we obtain Equation \eqref{eq:roll}, 
\be\label{eq:roll1}
\cos\left({\pi  w'/  n}\right)=\cos\left({\pi  w/  n}\right)
\left[1-4\sin^2\left({\pi  w/ n}\right)\sin^2\phi\right].
\ee

This completes the proof of part (b) of Theorem \ref{thm:regularpolys}, except for the bound $w<w'<n$. This follows from the last equation: by the periodicity of the left hand side, one can assume, without loss of generality, that $0\leq w'< n$. Now the right hand side  is a strictly decreasing function of $\phi\in (0,\pi/2)$, taking values in the open interval $(\cos(3\pi w/n), \cos(\pi w/n))$. Thus if $\cos(\pi w'/n)$ is one of these values we must have that $\pi w/n < \pi w'/n$, or $w<w'.$ 

\mn\paragraph{Proof of Part (a).} This proof is very similar to part (b)  above, except that for the  trivial (unlifted) monodromy condition on $\Gamma$, one requires only that  the lifted monodromy satisfies $g=\pm 1,$ so  in Lemma  \ref{lemma:p}  one can drop the requirement $w'\equiv w$ (mod 2). 

\mn\paragraph{Proof of Part (c).} Let $\Gamma'$ be the regular polygon traced on the moving sphere as it is rolled along $\Gamma$. Its vertices are $\v'_0, \ldots, \v'_{n-1}$, arranged on a circle of radius $\phi'$, with  $\theta'$ the angle of rotation at the center, sending $\v'_i\mapsto \v'_{i+1}.$ To show that $\w'$ in equation  \eqref{eq:roll} is the winding number of $\Gamma'$ about its center we thus need to show that $\theta':=2\pi w'/n$. 

We consider the shaded triangle Figure \ref{fig:trigo}  and the corresponding triangle on the moving sphere. On the stationary sphere the angles   are $\theta/2, \pi/2, A$, and the sides opposite the  first two angles are $\delta/2, \phi$ (respectively). On the moving sphere the angles are  $\theta'/2, \pi/2, A$, and the sides opposite the  first two angles are $3\delta/2, \phi'$ (respectively). By formulas (R3) and (R9) of \cite{W} applied to these two triangles, 
\begin{align*}
\cos(\theta/ 2)&=\sin A\cos (\delta/ 2),\\
 \cos(\theta'/2)&=\sin A\cos (3\delta/ 2),\\
 \sin(\delta/2)&=\sin(\theta/2)\sin\phi.
\end{align*}
Hence 
\begin{align*}\cos(\theta'/ 2)&=
\cos(\theta/ 2){\cos(3\delta/2)\over \cos(\delta/2)}=
 \cos(\theta/ 2)\left[1-4\sin^2(\delta/2)\right]\\
 &=\cos(\theta/ 2)\left[1-4\sin^2(\theta/2)\sin^2\phi\right]=\cos\left({\pi  w'/  n}\right).
\end{align*}
Thus $\cos(\theta'/ 2)=\cos\left({\pi  w'/  n}\right)$, hence $\theta'=2\pi w'/n$, as needed.

\mn\paragraph{Proof of Part (d).} We first  show  that for all $n\geq 6$, $w=2$ and $w'=4$, Equation \eqref{eq:roll1} has a solution, using an intermediate-value argument: let  $\alpha=2\pi/n,$ then  we need to show that $\cos(2\alpha)=\cos\alpha\left(1-4\sin^2\alpha\sin^2\phi\right)$ has a solution $\phi\in(0, \pi/2)$. As $\phi$ varies in $[0,\pi/2]$,  the right hand side of this equation decreases monotonically  from $\cos\alpha$ to $\cos\alpha\left(1-4\sin^2\alpha\right)=\cos(3\alpha)$. For $n\geq 6$ one has $3\alpha\leq \pi$, hence $\cos(\alpha)>\cos(2\alpha)>\cos(3\alpha)$, so there is an intermediate  value of $\phi\in(0,\pi/2)$ for which the right hand side is $\cos(2\alpha)$.

To show that there is  no solution of \eqref{eq:roll1}  with $n<6$, we can either use  Theorem  \ref{thm:intro1}, or more elementary, prove directly that there are no solutions to Equation \eqref{eq:roll1} with $n<6$, $0<w<n/2$,  $w<w'<n$ and $w'\equiv w$ (mod 2).
There are only 3 cases of $(n,w,w')$ satisfying these restrictions: $(4,1,3),(5,1,3)$ and $(5,2,4).$ In these 3 cases it is easy to show that the equation reduces to $\cos\phi=0,$ so there is no solution $\phi\in (0,\pi/2)$. 
\qed

\subsection{Theorem \ref{thm:diag}}\label{sec:diag}

We divide the proof as follows:

\begin{enumerate}[ (1)]
\item  $\iota:\Qd\to \Qo$ is an embedding, mapping $\Dd$ to $\Do$. 

\item $\Phi:\tQr\to\tQo$ is a diffeomorphism
\item $\Phi$ maps $\tDr$ to $\tDo$. 
\end{enumerate}

\mn\paragraph{Proof of (1).} This was shown in \cite[\S3.3]{BHN}, with slightly different notation, so we sketch the proof here. The map $\Rtt\to\P(\Im(\tO))$, $(\q,\p)\mapsto[\m],$ where  $\m=\left(\begin{smallmatrix}1&\q \\ \p&-1\end{smallmatrix}\right),$ is clearly injective (an affine chart). If $(\q,\p)\in \Qd$, i.e., $\p\q=1$,  then   $\langle\m, \m\rangle=-1+\b\A=0$, i.e. $\iota(\Qd)\subset \Qo$. 

Next let $\Omega:=\m\d \m$, an 
$\tO$-valued 1-form on $\Im(\tO)$. Explicitly, 
\be\label{eq:Omega}\Omega=
\left(\begin{array}{lr}
x\,\d x-\d\p\, \q  & 
x\,\d\q - \q \,\d x+\p\times \d\p \\  
 \p\, \d x- x\d\p + \q\times \d\q&
x\,\d x-\p \,\d\q
\end{array}\right).
\ee

Next one calculates  that $r^*\Omega =r^2\Omega$, $r\in\R^*,$  and that $\Omega$ vanishes along  the radial directions  in the null cone $C\subset \Im(\tO)$, hence the  restriction of $\Omega$ to $C$ descends to the quotient $(C\setminus 0)/\R^*=\Qo$, with kernel $\Do$ (see \cite[Proposition 3.5]{BHN} for the detailed calculation). It is thus enough to show that the kernel of the pull-back of $\Omega$ to $\Qd$ by $\tilde\iota:(\A,\b)\mapsto \left(\begin{smallmatrix}1&\q \\ \p&-1\end{smallmatrix}\right)$ is $\Dd$. Indeed, from Equation \eqref{eq:Omega} follows
$$\tilde\iota^*\Omega=\left(\begin{array}{lr}
-\d\p \, \q & 
\d\q+\p\times \d\p \\  
-\d\p + \q\times \d\q&
-\p \,\d\q
\end{array}\right),$$
then one checks  that the common kernel of the entries of this 1-form, restricted to $\Qd$, is indeed $\Dd$.\qed

\mn\paragraph{Proof of (2).} We express $\Phi:\tQr\to\tQo$ as the composition 
\be\label{eq:comp}
S^2\times S^3\xrightarrow{f}S^2\times S^3\xrightarrow{j}C\setminus 0\xrightarrow{\pi}\tQo,
\ee 
where \begin{itemize}
\item $f$ is the restriction to $S^2\times S^3$ of the map 
$$\Im(\H)\times\H\to \Im(\H)\times\H,\ \  (\v,q)\mapsto (\v, \v q),$$
 \item $j$ is the restriction to $S^2\times S^3$ of the linear isomorphism 
 \be\label{eq:iso}
  \Im(\H)\times\H\to\Im(\tO),\ \  (\v, x+\u)\mapsto \left(\begin{matrix}x&\v+\u\\ \v-\u&-x\end{matrix}\right),
 \ee
 where $\u, \v\in\Im(\H)$, $x\in\R$, 
\item $C=\{\m\in\Im(\tO)\st \langle\m,\m\rangle=0\}$ is the  {\em null cone} in $\Im(\tO)$,   and 
\item 
$\pi$ is the restriction to $C\setminus 0$ of the canonical projection 
 $$ \Im(\tO)\setminus 0\to (\Im(\tO)\setminus 0)/\R^+=S^6, \ \  \m\mapsto \R^+\m.
 $$
\end{itemize}

\begin{remark}
We note, as we did after Equation \eqref{eqn:Phi}, that for Equation 
\eqref{eq:iso} to make sense we need to identify $\R^3$ (column vectors) with $\Rts$ (row vectors) using the standard Euclidean inner product in $\R^3$. This convention will be kept implicitly  for the rest of the article.
\end{remark}

First,   $f$ is a diffeomorphism because it has an inverse, $$(\v,q)\mapsto (\v, -\v q).$$  

Next, we verify  that $j$ maps $S^2\times S^3$ into $C\setminus 0$: if $(\v, x+\u)\in S^2\times S^3$, then $|\v|^2=x^2+|\u|^2=1,$ hence if $\m= j(\v, x+\u)= \left(\begin{smallmatrix}x&\v+\u\\ \v-\u&-x\end{smallmatrix}\right)$ then $\langle\m, \m\rangle=-x^2+(\v+\u)\cdot(\v-\u)=|\v|^2-(x^2+|\u|^2)=1-1=0,$ hence $ j(\v, x+\u)\in C$. Now $0$ is clearly not in the image, since $j$ is the restriction of an injectve linear map and $0$ is not in $S^2\times S^3$. 

Next, to show that  $\pi\circ j$ is a diffeomorphism, we show that it is bijective and a local diffeomorphism. 

To show that $\pi\circ j$ is injective, note that $j$ is injective and its image intersects each of the fibers of $\pi$  at a single point. 

To show that $\pi\circ j$ is surjective  let  $[\m]\in\tQo$, then 
$\m=\left(\begin{smallmatrix}x&\A\\ \b&-x\end{smallmatrix}\right)\in C\setminus 0$, hence $ -x^2+\b\A=0.$ Let $\v:=(\A+\b)/2, \u:=(\A-\b)/2,$ then $|x+\u|^2=x^2+|\u|^2
=x^2+|\A-\b|^2/4=x^2+|\A+\b|^2/4-\b\A=|\v|^2.$ Thus $|x+\u|=|\v|\neq 0,$ so we can  (positively) rescale $\v$ and $x+\u$ simultaneously to unit vectors. 

To show that $\pi\circ j$ is a local diffeomorphism it is enough to show, by the inverse function theorem, that its differential $(\pi\circ j)_*$ is bijective at each point of $S^2\times S^3$. 
Now  $j$ is the restriction of a linear isomorphism, hence is an immersion, i.e., $j_*$ is injective at each point of $S^2\times S^3$, and  $\pi$ is a submersion, i.e., $\pi_*$ is surjective at each point of $C\setminus 0$, with kernel in the radial direction, transverse to the image of $j_*$. Hence the composition $\pi_*\circ j_* =(\pi\circ j)_*$ is bijective, as needed. \qed

\mn\paragraph{Proof of (3).} We shall prove this statement in several steps:
\begin{enumerate}[(a)]
\item Define a transitive action of $K:=S^3\times S^3$ on $\tQr$ and  $\tQo$. 

\item  $\Phi:\tQr\to\tQo$  is $K$-equivariant.
\item $\tDr,\tDo$ are $K$-invariant. 
 \item $\Phi_*$ maps  $\tDr$ to $\tDo$ at a single point of $\tQr$. 
\end{enumerate}

We proceed with proofs of each of these steps. 

\mn\paragraph{Step (a).}Define a linear action of $K=S^3\times S^3$ (pairs of unit quaternions)   on $\Im(\H)\oplus\H$ via   
\be\label{eq:Kaction}
(q_1,q_2):(\v,q)\mapsto (q_1\v\bar q_1, q_1q\bar q_2).
\ee
This leaves invariant $\tQr=S^2\times S^3$ so defines a $K$-action on it, clearly transitive. 
 
This linear $K$-action on $\Im(\H)\oplus\H$ also induces a $K$-action on $\Im(\tO)$ via the linear isomorphism $\Im(\H)\oplus\H\to\Im( \tO)$ of formula \eqref{eq:iso}. 
Under this isomorphism, $\langle\ ,\ \rangle$ becomes the  quadratic form $|\v|^2-|q|^2$ on $\Im(\H)\oplus\H$, which is clearly $K$-invariant, hence the null cone $C\subset \Im(\tO)$ is $K$-invariant, inducing a $K$-action on $\tQo=(C\setminus 0)/\R^+.$

\mn\paragraph{Step (b).}Each of the maps $f,\iota, \pi$ of \eqref{eq:comp} are  $K$-equivariant, hence so is $\Phi$. The (easy) verification is  left to the reader. 


\mn\paragraph{Step (c).}\label{stepc} 
We will use throughout  a well-known (and easy to verify) quaternion identity: 
$$\v\u=-\v\cdot\u+\v\times\u, \ \   \forall\v,\u\in\Im(\H).
$$

We start with the $K$-invariance of $\tDr$. Note that the $K$-action on 
$\tQr$ commutes with that of $(\pm 1, \pm 1)$, 
hence it descends to an action of $K/(\pm 1,\pm 1)=\SO_3\times \SO_3$ on the quotient 
$\Qr=
\tQr/(1,\pm1)=
S^2\times \SO_3,$ 
$(g_1, g_2): (\v, g)\mapsto (g_1 \v, g_1gg_2^{-1}).$
 Thus, in order to show that $\tDr$ is $K$-invariant it is enough to show that $\Dr$ is $\SO_3\times \SO_3$-invariant. Now  at a point $(\v,g)\in\Qr=S^2\times \SO_3$, $\Dr$  consists of vectors $(\dot\v, \dot g)\in T_{(\v, g)}\Qr$ satisfying  
\be\label{eq:roll1}4\dot\v=\bo \times  \v, \ \ \bo\cdot\v = 0,\ee
where $\bo\times \x=\dot gg^{-1}\x,$ see Definition \ref{def:rol}. Now $(g_1, g_2)_*:
(\dot\v, \dot g)\mapsto (g_1\dot\v, g_1\dot gg_2^{-1}),$ thus 
$$g_1\dot gg_2^{-1}(g_1gg_2^{-1})^{-1}\x=g_1\dot gg^{-1}g_1^{-1}\x=g_1(\bo\times g_1^{-1}\x)=(g_1\bo)\times \x,$$
hence $(g_1, g_2)_*:(\dot\v, \bo)\mapsto (g_1\dot\v, g_1\bo).$ Now if $(\dot\v, \bo)$ satisfy  Equations \eqref{eq:roll1} then $4g_1\dot\v=4g_1(\bo \times  \v)=4(g_1\bo) \times ( g_1\v),$ $(g_1\bo)\cdot (g_1\v) =\bo\cdot \v=0,$ hence $(g_1\dot\v, g_1\bo)$ satisfy them as well. This shows that $\tDr$ is $K$-invariant.

To show that $\tDo$ is $K$-invariant it is enough to show that the $K$-action defined on $\Im(\tO)$ by the isomorphism $\Im(\H)\oplus\H\to\Im(\tO)$ of formula \eqref{eq:iso} 
preserves octonion multiplication. For this, we show that the image of the  infinitesimal $K$-action on $\Im(\tO)$ is contained in the Lie algebra $\g_2\subset\End(\Im(\tO))$ defined by Equation \eqref{eq:deriv}.

To this end,  let $(\v_1, \v_2)\in \Im(\H)\oplus\Im(\H)$, thought of as the Lie algebra of $K=S^3\times S^3$. The infinitesimal action of $(\v_1, \v_2)$ on $\Im(\H)\oplus\H$, corresponding to the $K$-action of Equation \eqref{eq:Kaction}, is
$$(\v, q)\mapsto (2\v_1\times \v, \v_1q-q\v_2).
$$
Conjugating this action with the ismorphism $\Im(\H)\oplus\Im(\H)\to\Im(\tO)$ of Equation \eqref{eq:iso}, we obtain the infinitesimal action of $(\v_1, \v_2)$ on $\Im(\tO)$, 
$$
\left(\begin{array}{lr}
x&\A\\ 
\b&-x 
\end{array}\right)\mapsto \left(\begin{array}{lr}
\tilde x&\tilde\A\\ 
\tilde\b&-\tilde x 
\end{array}\right),$$
where 
$$\left\{\begin{array}{rl}
\tilde\A&={1\over 2}(3\v_1+\v_2)\times\A+{1\over 2}(\v_1-\v_2)\times\b+(\v_1-\v_2)x \\
&\\
\tilde\b&= {1\over 2}(\v_1-\v_2)\times\A+{1\over 2}(3\v_1+\v_2)\times\b+(\v_2-\v_1)x\\
&\\
\tilde x& ={1\over 2}(\v_2-\v_1)\cdot(\A-\b)
 \end{array}\right.
 $$
%
(to simplify notation, all vectors in this formula are column vectors).  One can see easily that this is the action of $\rho(T, \bQ, \bp)/2$ of Equation \eqref{eq:deriv}, where 
\be\label{eq:emb}T\x=(3\v_1+\v_2)\times\x,\ \bQ=\v_1-\v_2,\ \bp=\v_2-\v_1.
\ee 
This concludes the proof of Step (c). \qed

\mn\paragraph{Step (d).}  We  first find  equations for $\tDr$ on $\tQr=S^2\times S^3\subset \Im(\H)\oplus\H$, using coordinates $\v\in \Im(\H),$ $s+\w\in\H$. 

\begin{lemma}  $\tDr\subset T\tQr$ is given by 
\begin{align*}
&2\d\v+\v\times (s \,\d\u-\u \,\d s+\u\times \d\u)=0&&\mbox{ \em  (no slip),}\\
&\v\cdot (s \,\d\u-\u \,\d s+\u\times \d\u)=0&&\mbox{ \em (no twist).}
\end{align*}
\end{lemma}
\begin{proof}
Recall  first the equations that define $\Dr$ (Definition  \ref{def:rol}, for $\rho=3$):
$$4\dot\v=\bo \times  \v \mbox{ (no slip),}
\quad \bo\cdot\v = 0 \mbox{ (no twist),}
$$
where $\bo$ is defined by $\dot gg^{-1}\x=\bo\times\x.$ 
\newcommand{\X}{{\bf X}}

Now let $q=q(t)\in S^3$, then $1=q\bar q$ implies $0=\dot q \bar q+q\dot{\bar q},$ i.e., $\dot q\bar q\in \Im(\H).$  
Next define $g=\Ad(q)$, i.e.,  $\x=g\X=q\X \bar q,$ where $\X\in\R^3$ (fixed). 
Then, on the one hand, 
$$\dot\x=\dot g\X=\dot g g^{-1}\x=\bo\times \x,
$$
and on the other hand, 
$$\dot \x=\dot q\X \bar q+q\X\dot{\bar q}=
\dot q\bar q \x q\bar q+q\bar q \x q\dot{\bar q}=
\dot q\bar q \x - \x \dot q{\bar q}=
2\dot q\bar q\times \x,$$
hence  $\bo=2\dot q\bar q$. 

Next, writing $q=s+\u$, one has $1=q\bar q=s^2+|\u|^2,$ hence 
$0=s\dot s +\u\cdot\dot\u$, thus
\begin{align}\label{eq:dq}
\begin{split}
\dot q\bar q
&=(\dot s+ \dot \u)(s-\u)
=\dot s s - \dot s \u +\dot\u s -\dot\u\u\\
&=-\dot\u\cdot \u- \dot s \u +\dot\u s -(-\dot\u\cdot\u+\dot\u\times\u)\\
&=s\dot \u-\u\dot s +\u\times \dot\u.
\end{split}
\end{align}

It follows that  the no slip equation on $\Qr$, $4\dot\v=\bo\times \v$, pulled back to $\tQr$ by $q\mapsto \Ad(q)$, is
$$2\dot\v+\v\times \dot q\bar q=2\dot\v+\v\times(s\dot \u-\u\dot s +\u\times \dot\u)=0.$$

The no-twist equation also follows from the above expression \eqref{eq:dq} for $\dot q\bar q$,
$$\dot q\bar q\cdot\v=\v\cdot(s\dot \u-\u\dot s +\u\times \dot\u)=0,$$
as needed. 
\end{proof}

\mn{\em Proof of step (d) continued.} Now fix the point $(\i,1)\in \tQr$. From the last lemma follows that $\tDr$ at this point is given by 
\be\label{eq:droll}
\d w_1=\d  v_1=2\d {v_2}-\d  w_3=2\d  v_3+\d  w_2=0.
\ee

Next consider $\Phi(\i,1)=[{0\ 2\i\choose 0\ 0}]\in \tQo.$ According to Equation \eqref{eq:Omega}, at this point $\tDo$, pulled-back to the tangent to  $C\subset \Im(\H)$ at ${0\ 2\i\choose 0\ 0}$,  is given  by 
\be\label{eq:doct}
\d x=\d{ b}_1= \d{ A}_2=\d{ A}_3=0.
\ee
Now $\Phi=\pi\circ \iota\circ f$, and $\iota\circ f$  is given by $(\v,s+\u)\mapsto {x\quad \,\,\A\choose \,\,\b \quad -x}$, where 
$$
\A=(1+s)\v+\v\times \u,\ \b=(1-s)\v-\v\times\u,\ x=-\v\cdot\w.
$$
The pull back of  Equations \eqref{eq:doct} under this map are  
$$\d w_1=\d s=2\d v_2-\d w_3=2\d v_3 +\d w_2=0.$$
This coincides with Equations \eqref{eq:droll}, modulo the  tangency  equations to $S^2\times S^3$ at $(\i,1)$,  $\d v_1=\d s=0.$
\qed

\subsection{Theorem \ref{thm:main}}\label{ss:pfmain}
The proof is based on Theorem \ref{thm:diag}. We need to establish  first two lemmas. 

To state the  first lemma, we consider an ordered pair of points 
$\left([\v_1], [\v_{2}]\right)\in \left(S^2/\pm 1\right)\times \left(S^2/\pm 1\right),$ and  associate with it a rolling monodromy $\mu\in S^3$, as follows. 
%
Consider a simple geodesic segment from  $[\v_1]$ to $[\v_2]$ (there are 2 such segments), lift it to a spherical segment in $S^2$ (there are 2 such lifts), roll the moving sphere along this segment, then take the resulting  lifted rolling monodromy $\mu\in S^3$. 

\begin{lemma}\label{lemma:lifts}The lifted rolling monodromy $\mu\in S^3$ depends only on the ordered  pair $\left([\v_1], [\v_{2}]\right)\in \left(S^2/\pm 1\right)\times \left(S^2/\pm 1\right)$ and not on the various choices made. 
\end{lemma}

\begin{proof}Let $\delta:=$dist$([\v_1], [\v_{2}]),$ $0<\delta\leq\pi/2. $  There are two directed geodesic segments in $S^2/\pm 1$ connecting $[\v_1]$ to $ [\v_{2}]$, of lengths $\delta, \pi-\delta.$ Each has two possible lifts  to geodesic segments in $S^2$, of the same length, a total of 4 possibilities, with endpoints (i) $\v_1, \v_{2}$, (ii) $\v_1, -\v_{2}$, (iii)  $-\v_1, -\v_2$, (iv) $-\v_1, \v_2$, of lengths $\delta, \pi-\delta, \delta, \pi-\delta $ (respectively). See Figure \ref{fig:lifts}. 

\begin{figure}[h!]
\centering
\def\svgwidth{.35\textwidth}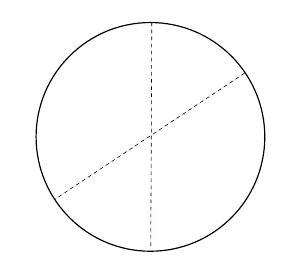
\caption{\ }
\label{fig:lifts}
\end{figure}

Let us take case (i).  Let $\w=\v_1\times\v_{2}/\|\v_1\times\v_{2}\|.$ As we roll the small sphere along the spherical arc segment of length $\delta$ from
 $\v_1$ to $\v_{2}$, it rotates about the $\w$ axis by an angle of $4\delta$. Thus $\mu=e^{2\w\delta}.$ In case (ii), the arc length changes to $\pi-\delta$ and $\w$ changes to $-\w$. Thus $\mu=e^{2(\pi-\delta)(-\w)}=e^{2\w\delta}$, same as in case (i). Cases (iii) and (iv) are analyzed similarly. 
\end{proof}

The  second lemma concerns the map $\Phi:\tQr\to\tQo$ of Equation \eqref{eqn:Phi}. 

\begin{lemma}\label{lemma:Phi}
Let $(\v,q)\in\tQr$, $[\m]=\Phi(\v,q)\in \tQo.$ Then 
$\Phi(-\v,q)=[-\m].$ 
\end{lemma}
In other words, $\Phi$ is $\Z_2$-equivariant, with respect to the actions $(-1) \times 1$, $- 1$ on $\tQr, \tQo$ (respectively). 
The proof is immediate from formula \eqref{eqn:Phi}. \qed

\mn

Now we proceed to the bijective correspondence indicated in the statement of Theorem \ref{thm:main}. We start with  an equivalence class $[\Gamma]$ of non-degenerate closed spherical $n$-gons with trivial lifted rolling monodromy and an element $q\in S^3$. We  associate with  $([\Gamma], q)$  a non-degenerate closed horizontal $n$-gon in $\Qo$, as follows. $[\Gamma]$ is given by the projected vertices $[\v_1], \ldots, [\v_n]\in S^2/\pm 1$ (see the  first paragraph after Definition \ref{def:equiv}).  We pick simple edges between succesive vertices (there are two possibilities for each successive pair) and obtain a closed polygonal path in $S^2/\pm 1$. This path is then lifted to a spherical $n$-gon  $\Gamma$ in $S^2$ (there are two such lifts), with vertices $\v_1, \v_2, \ldots,$ closed up to sign, i.e., $\v_{i+n}=\pm\v_i$ (same sign for all $i$). Let $\mu_i\in S^3$ be the lifted rolling monodromy along the edge of $\Gamma$ from $\v_i$ to $\v_{i+1}$. By Lemma \ref{lemma:lifts}, $\mu_i$ depends only on $[\v_i], [\v_{i+1}]$. 

Next we lift $\Gamma$ horizontally to $\tQr$, starting at $(\v_1,q)$, and obtain a horizontal polygon, with vertices $(\v_1, q_1),  \ldots (\v_n, q_n)$, where $q_1:=q, q_2:=\mu_1 q_1,\ldots, q_n=\mu_{n-1}q_{n-1}.$ This polygon is closed up to sign of $\v_i$, i.e., $(\v_{i+n}, q_{i+n})=(\pm \v_i, q_i)$ (same sign for all $i$, depending on the choices made along the way). 																						

This horizontal polygon is mapped by $\Phi$ to a  horizontal polygon in $\tQo$ with vertices $[\m_i]:=\Phi[(\v_i,q_i)],$ $i=1, \ldots, n.$ By Lemma \ref{lemma:Phi}, this polygon is closed up to sign, i.e., $[\m_{i+n}]=[\pm\m_i]$ (same sign for all $i$). The projection of this polygon to  $\Qo$ is thus closed and horizontal, and its vertices do not depend on the choices made along the way (the edges do depend on the initial choice of edges in $S^2/\pm 1$, but a dancing pair of polygons is specified by its vertices alone). See figure \ref{fig:corresp}.

\begin{figure}[h!]
\centering
\def\svgwidth{.9\textwidth}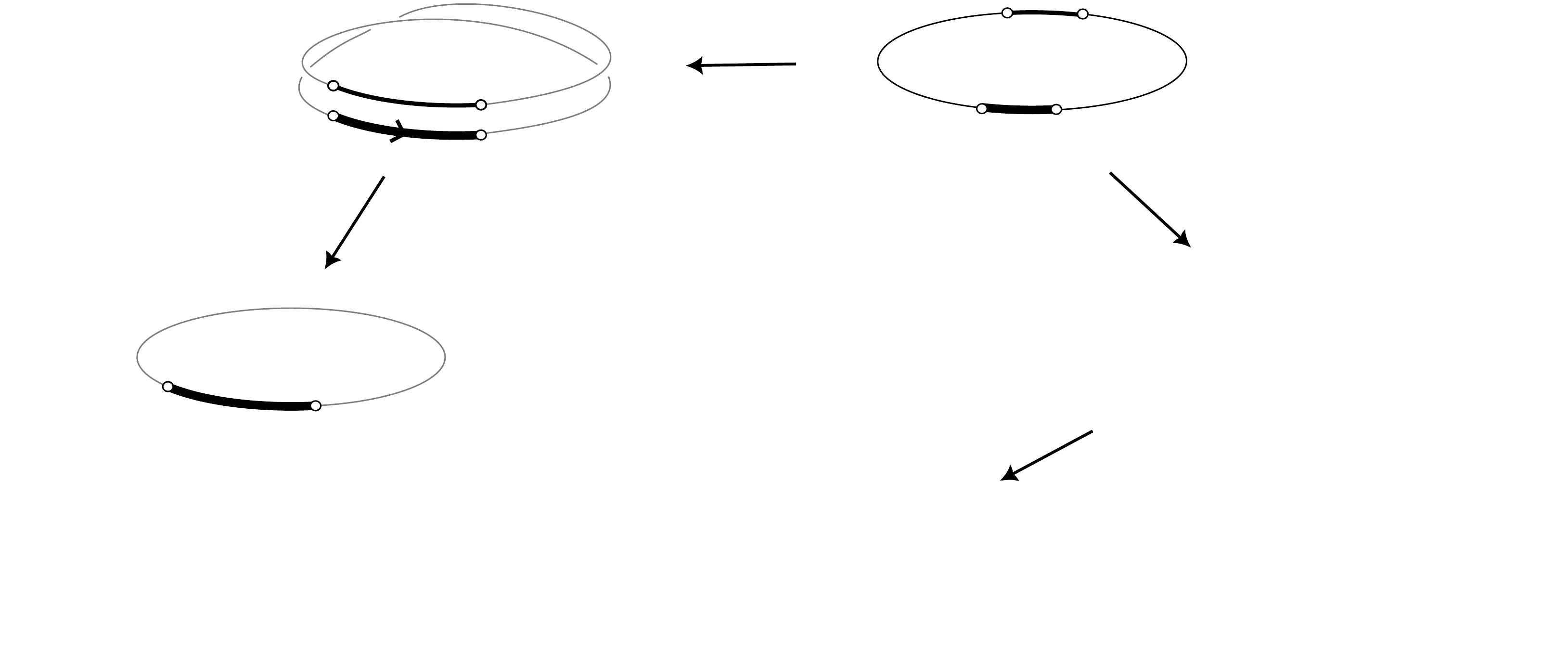
\caption{\ }
\label{fig:corresp}
\end{figure}

 If we add the genericity condition to $([\Gamma], q)$, then the $n$ vertices in $\Qo$ lie in $\Qo_*=\iota(\Qd)$, i.e., correspond to a dancing pair of polygons as per Theorem  \ref{thm:horpol}. 

This correspondence is clearly invertible. Starting with a dancing pair of closed polygons, we first lift it to a horizontal polygon in $\Qd$  (by Theorem \ref{thm:horpol}), then map it to  a closed horizontal polygon in $\Qo_*$ (by step (1) of the proof of Theorem \ref{thm:diag}, see Section \ref{sec:diag}),  lift to a  horizontal polygon in  $\tQo_*$, then map by $\Phi^{-1}$ to a horizontal polygon in $\tQr_*$, with vertices 
$(\v_1, q_1), \ldots, (\v_n, q_n)$, closed up to sign of $\v_i$. It follows that the spherical polygon $\Gamma$ with vertices  $\v_1, \ldots, \v_n$  has trivial lifted rolling monodromy and that $([\Gamma], q_1)$ is generic. \qed

\subsection{Theorem  \ref{thm:intro1}}

From the correspondence of Theorem \ref{thm:horpol} (proved below in Section \ref{sec:horpol}), it is enough to  show that (1)  there exist non-degenerate horizontal $n$-gons in $\Qd$  for every $n\geq 6$, and (2) every horizontal $n$-gon in $\Qd$ for $n\leq 5$ is degenerate. The  first statement follows from Corollary \ref{cor:regularpolys} and Theorem \ref{thm:main}. We procced to prove the  second statement.

Let $q_1, q_2, q_3\in\Qd$ be the vertices of an open  non-degenerate  horizontal 3-gon; i.e., the three  points are distinct and the lines $q_1q_2$ and $q_2q_3$ are horizontal and distinct. We shall prove that 
\begin{enumerate}[{\rm (1)}]
\item $q_1q_3$ is not  horizontal; therefore, there are no non-degenerate horizontal triangles.
\item If $q_4\in\Qd$ is such that $q_1q_4$ and $q_3q_4$ are horizontal, then $q_4=q_2$; therefore, there are no non-degenerate horizontal quadrilaterals.
\item If $q_4, q_5\in\Qd$ are such that $q_3q_4$, $q_4q_5$ and $q_5q_1$ are horizontal, then either $q_4=q_2$ or $q_5=q_2$; in either case, the pentagon will be a degenerate one.
\end{enumerate}

We proceed with proofs of (1) and (3). The proof of (2) is similar to (3), slightly simpler, and is omitted. 

\mn\paragraph{Proof of (1).} Suppose $q_1,q_2,q_3$ are 3 distinct points in $\Qd$ forming a horizontal triangle. We show that it is degenerate, i.e., the 3 points are collinear. Let $q_i=(\A_i,\b_i)$,  $i=1,2,3$.  Then, by Definition \ref{def:hor}, 
\be\label{eq:hors}
\b_i-\b_j=\A_j\times \A_i, \ i,j=1,2,3.
\ee

We can assume  that $\A_1, \A_2, \A_3$ are distinct:
 if  $\A_i=\A_j$, $i\neq j$,  then, by Equation \eqref{eq:hors}, $\b_i-\b_j=\A_j\times \A_i=0$, hence $q_i=q_j$, a contradiction.  

Next, again by \eqref{eq:hors}, 
\begin{align*}
(\A_1-\A_3)\times (\A_2-\A_3)&=\A_1\times\A_2-\A_1\times\A_3-\A_3\times\A_2=\\
&=
(\b_2-\b_1)-(\b_3-\b_1)-(\b_2-\b_3)=0,
\end{align*}
so $\A_1-\A_3=\lambda(\A_2-\A_3)$ for some $\lambda\neq 0.$ Taking the cross product with $\A_3$,  $\A_1\times \A_3=\lambda(\A_2\times\A_3)$, which implies, by \eqref{eq:hors}, $\b_1-\b_3=\lambda(\b_2-\b_3)$, so $q_1-q_3=\lambda(q_2-q_3)$, hence  $q_1, q_2, q_3$ are collinear. 

\mn\paragraph{Proof of (3).} Let $\e_1, \e_2, \e_3$ be the standard basis of $\R^3$ and $\e^1, \e^2, \e^3$ the dual basis. Since $\SL_3(\R)$ acts transitively on $\Qd$ preserving horizontality,  we may assume, without loss of generality, that $q_2=(\e_1,\e^1).$ The  isotropy group at this point  is 
$$\left\{ \begin{pmatrix} 1 & 0\\
			  0 & A\end{pmatrix} \left| \right. A\in \SL_2(\R)\right\},$$
acting  on the affine plane $ \Dd_{q_2}=\mbox{Span}\{ (\e_2, \e^3), (\e_3, -\e^2)\}$. Hence, without loss of generality, we may assume that \begin{align*}
q_1&=q_2+(\e_2, \e^3)=(\e_1+\e_2,\e^1+\e^3), \\  
q_3&=q_2+a(\e_3, -\e^2)=(\e_1+a\e_3,\e^1-a\e^2),
\end{align*}
 for some $a\neq 0$.

In a similar fashion, we see that any point $q_5\in\Qd$ that may be joined horizontally to $q_1$ must be of the form 
$$q_5=q_1+\left(x\e_1+y\e_2-x\e_3, -x\e^1+x\e^2+(y-x)\e^3\right),$$ for some choice of $x,y\in\R$. Likewise, any point $q_4\in\Qd$ that may be joined horizontally to $q_3$ must be of the form 
$$q_4=q_3+(ab\e_1+b\e_2+c\e_3,-ab\e^1+(a^2b-c)\e^2+b\e^3),$$ for some $b,c\in\R$.

Assuming that $q_4q_5$ is a horizontal line, the horizontality equation $\b_5-\b_4=\A_4\times \A_5$ gives the following system
$$
\left\{ \begin{array}{rcl} (a+x)b+(1+y)c&=&x-a(1+y)\\
			     a(a+x)b+xc&=&ax\\
			     b\left[ x-a(1+y)\right]&=&x
			     \end{array}\right. 	.
$$Solving for $b,c$ in terms of $x,y$ in the first two equations and then using the third, implies that:
\begin{itemize}
\item if $x\neq 0$ $\Rightarrow$ $a=0$, i.e., $q_3=q_2$, a contradiction;
\item if $x=0$ $\Rightarrow$ $b=0$ $\Rightarrow$ either $c=-a$ (which implies $q_4=q_2$) or $y=-1$ (which implies $q_5=q_2$). 
\end{itemize}		
This concludes the proof of (3). 
\qed	  


\appendix

\section{$\G$-symmetry}\label{app:sym}

The definition  of the  \CED\ using split octonions $\tO$ (see Section \ref{ss:oct}) gives rise to an action of the automorphism group of this algebra, $\G:={\rm Aut}(\tO)$, as the (local)  symmetry group of the various models of this distribution that appear in this article.  In this appendix we collect some  explicit formulas for the associated infinitesimal action. Most of the material here appeared before, e.g., in  \cite{BM, BHN, BH}. 

\subsection{$\g_2$} 

The  Lie algebra of $\G$ is $\g_2\subset \so_{4,3}\subset\so_{4,4}$, where $\so_{4,4}$ is the Lie algebra of the subgroup of $\GL(\tO)$ preserving the  inner product on $\tO$ of Equation \eqref{eq:ip}, and $\so_{4,3}$ is the subalgebra preserving the spltting $\tO=\Re(\tO)\oplus\Im(\tO)$.

Now  $\g_2$  is the algebra of {\em derivations} of $\tO$,  i.e., maps $D\in \End(\tO)$ satisfying 
$$D(\m \m')=(D\m) \m'+\m (D\m')$$
for all $\m, \m'\in \tO.$ It preserves the decomposition $\tO=\Re(\tO)\oplus\Im(\tO)$ and acts trivially on the  first summand. On the  second summand we have the infinitesimal action of  $\g_2$, defined as follows. 

Consider  the map  $\rho:\sl_3(\R)\times\R^3\times(\R^3)^*\to\End(\Im(\tO))$, 
\be\label{eq:deriv}
 \rho(T, \bQ,\bp):
\left(\begin{matrix} x& \A \\ \b& -x
\end{matrix}\right)
\mapsto 
\left(\begin{matrix}
 \bp\A+\b\bQ & T\A-\bp\times \b+2\bQ x\\ 
\bQ\times \A-\b T+2\bp x& - \bp\A- \b\bQ
 \end{matrix}\right),
\ee
where $x\in\R$,  $\A,\bQ\in\R^3$ (column vectors), $\b, \bp \in \Rts$ (row vectors) and $T\in\sl_3(\R)$ (traceless $3\times 3$ real matrices).   

Now $\rho$ is  clearly an injective  linear map, hence its image is a 14-dimensional subspace  of $\End(\Im(\tO))$. One can check that 
$[\rho(T_1,\bQ_1,\bp_1),\rho(T_2,\bQ_2,\bp_2)]=\rho(T_3,\bQ_3,\bp_3),$
where
\begin{align*}
\label{eq:lie}
T_3&=[T_1,T_2]+3\left(\bQ_1\bp_2-\bQ_2\bp_1\right)+ \left(\bp_1\bQ_2-\bp_2\bQ_1\right){\rm I}_3,\\
\bQ_3&=T_1\bQ_2-T_2\bQ_1-2\bp_1\times \bp_2,\\
 \bp_3&=\bp_1T_2-\bp_2T_1+2\bQ_1\times \bQ_2.
\end{align*}
It follows that  the image of $\rho$ is  a 14-dimensional  Lie subalgebra of $\End(\Im(\tO))$. 

\begin{theorem}
The image of $\rho$ in $\End(\Im(\tO))$, extended to $\tO=\Re(\tO)\oplus\Im(\tO)$ by the $0$ action on $\Re(\tO)$, is the Lie algebra $\g_2\subset\End(\tO)$ of derivations of $\tO$. 
\end{theorem}

 See for example \cite[page 143]{J}. 
 
\subsection{The $\g_2$-action on $\tQr$ and $\Qd$}
 \label{app:emb}
 
 The $\G$-action on $\Im(\tO)$, given (infinitesimally) by Equation \eqref{eq:deriv},  induces actions on $\tQo$, $\Qo$,  and a local action on $\Qd$, then via diagram \eqref{eq:diag} an action on $\tQr$.  The latter action does {\em not} descend to $\Qr$, even infinitesimally (see next subsection).  The next two propositions give explicit formulas for the infinitesimal actions  on $\Qd, \tQr$. We start with $\Qd$. 

\begin{prop}
For each $(T, \bQ, \bp)\in \sl_3(\R)\times\Rt\times\Rts$, the infinitesimal action of $\rho(T, \bQ, \bp)\in \g_2$ on $\Qd\subset\Rtt$ is given by the vector field $f\partial_\A+g\partial_\b$ on $\Rtt$,  where 
$$\begin{array}{lcl}
 f(\A,\b)&=&2\bQ + T\A  - \bp\times\b-(\b\bQ+\bp \A)\A, 
\\ 
g(\A,\b)&=&2\bp-\b T  + \bQ\times\A-(\b\bQ+\bp\A)\b.
\end{array}
$$
\end{prop}

\newcommand{\bu}{{\bf u}}

\begin{proof}This appeared in \cite{BHN}, with somewhat different notation, so we will give it here again for completeness.

 We consider the  coordinates $(x,q)$ on $\Im(\tO)$, where $q=(\A,\b)\in\Rtt$. 
A linear vector field on $\Im(\tO)$, such as the one given by Equation \eqref{eq:deriv},  can be written  as  
$$\left(ax+  b q\right)\partial_x +\left( cx +dq\right)\partial_q, 
$$
where $a\in\R, b \in(\Rtt)^*,$ $ c\in\Rtt,$ $ d\in\End(\Rtt).$
It induces a vector field on $\P(\Im(\tO))$, given in the affine chart $\Rtt\to \P(\Im(\tO))$, $q\to [1,q]$, by 
\be\label{eq:pe}
\left[ c+(d-a)u- (bq)q\right]\partial_q.
\ee

For the linear vector field given by Equation \eqref{eq:deriv} one has
\begin{align*}
a&=0, &b(\A,\b)&=\bp\A+\b\bQ,\\
c&=2(\bQ, \bp),&d(\A, \b)&=(T\A-\bp\times\b, -\b T+\bQ\times \A).
\end{align*}
Using these in Equation \eqref{eq:pe} gives the stated formulas. 
\end{proof}

The formulas for the  $\g_2$-action  on $\tQr$ are more complicated. We shall treat a special representative  case with enough detail so that the interested reader can easily derive the formulas of the general case. 

\begin{prop} Let $\Xo$ be the linear vector field on $\Im(\tO)$ given by formula \eqref{eq:deriv} with  $\bQ=\bp=0$ and  $T=T^t$. The induced  infinitesimal action on $\tQr=S^2\times S^3\subset \Im(\H)\times \H$  is given by the vector field 
$$\Xr=f\partial_\v + g\partial_\w+h\partial_s,
$$ where $\v\in\Im(\H),\ s+\w\in\H$ are the standard Euclidean coordinates on $\Im(\H)\times\H$, and 
\begin{align*}
f&=T(s\v+\v\times\w)-\left[T(s\v +\v\times\w)\cdot \v\right]\v,\\
g &=\left[(s^2-1)\v+s(\v\times\w)\right] \times T\v +(s\v+\v\times\w)\times T(\v\times\w) \\
&\qquad\qquad\qquad +(\v\cdot \w)T(s\v +\v\times\w)
-2\left( T(s\v +\v\times\w)\cdot \v\right) \w, \\
h&=T\left(\v+s(\v\times \w)\right)\cdot\v -sT(s\v+\v\times\w)\cdot \v +T(\v\times\w)\cdot (\v\times\w).
\end{align*}
\end{prop}

\begin{proof}

The linear vector field $\Xo$  on $\Im(\tO)$ 
given by  Equation \eqref{eq:deriv} induces, by `spherization' ($\R^+$-quotient), a vector field on $\tQo$ and a vector field $\Xr$ on $\tQr$ via the diffeomorphism  $\Phi:\tQr\to \tQo$ of diagram \eqref{eq:diag}. 
Now $\Phi$  factors as a composition $\tQr\to
 C\setminus 0\xrightarrow{\pi}\tQo$ (recall that $C$ is the null cone in $\Im(\tO)$; see Equation \eqref{eq:comp}). The map $\tQr\to
 C\setminus 0$ is the restriction to $S^2\times S^3$ of  the map $\varphi:\Im(\H)\times\H\to\Im(\tO)$  given by Equations \eqref{eq:comp}-\eqref{eq:iso},  $(\v,  s+\w)\mapsto\left(\begin{matrix} x& \A \\ \b& -x
\end{matrix}\right)$,  where 
\begin{align}\label{eq:Abx}
\begin{split}\A&=(1+s)\v +\v\times\w,\\
\b&=(1-s)\v -\v\times\w,\\
x&=-\v\cdot\w. 
\end{split}
\end{align}
The inverse map is given by $\psi:\Im(\tO)\to\Im(\H)\times\H$,  
\begin{align}\label{eq:vws}
\begin{split}
&\v={\A+\b\over 2}, \\
&\w={\A\times \b\over 2} -x\v, \\
&s={1\over 4}(|\A|^2-|\b|^2),
\end{split}
\end{align}
with derivative
\begin{align*}
&\psi^*\left(\d\v\right)={1\over 2}\left(\d\A+\d\b\right),\\
&\psi^*(\d\w)={1\over 2}\left[\A\times \d\b -\b\times \d\A
-x(\d\A+\d\b)-(\A+\b)\d x\right],\\
&\psi^*(\d s)={1\over 2}\left(\A\cdot\d\A- \b\cdot\d \b\right).
\end{align*}
That is, $\varphi$ and $\psi$, when restricted to $S^2\times S^3$ and $M:=\varphi(S^2\times S^3)\subset C$ (respectively), are inverse maps: 
$$\begin{array}{ccc}
\Im(\H)\times\H& \begin{array}{c}\xrightarrow{\qquad \varphi \qquad}\\
\xleftarrow[\qquad \psi \qquad]{}\end{array}&
\Im(\tO)\\
\cup&&\cup\\
&&\\
S^2\times S^3&\cong&M
\end{array}
$$

Next let 
$$\Xo=\alpha\partial_\A+\beta\partial_\b+\gamma\partial_x, 
$$
and let   
\be\label{eq:euler}E:=\A\partial_\A+\b\partial_\b+x\partial_x
\ee 
be the Euler field on $\Im(\tO)$, which is  tangent to $C$. At each point of $M\subset C$, since $\Xo$ is tangent to $C$,  one can decompose uniquely $$\Xo=\Xo_\|+\lambda E,$$ where $\Xo_\|$ is tangent $M$ and $\lambda\in\R$. 
Then $\Xr=\psi_*(\Xo_\|).$

Next we note that, by Equations \eqref{eq:Abx},  $M$ is given in $C$ by the equation $|\A+\b|^2=4$, so $TM$ is the kernel of $(\A+\b)\cdot(\d\A+\d\b)$ restricted to $TC$. It follows that 
\begin{align*}0=(\A+\b)\cdot(\d\A+\d\b)(\Xo-\lambda E)&=
(\A+\b)\cdot (\alpha+\beta-\lambda(\A+\b))\\
&=(\A+\b)\cdot (\alpha+\beta)-4\lambda,
\end{align*}
thus 
$$\lambda={1\over 4}(\A+\b)\cdot(\alpha+\beta)={\v\over 2}\cdot (\alpha+\beta).
$$
It follows that 
\begin{align*}
f&=\d\v(\Xr)=\d\v\left(\psi_*\Xo_\|\right)=(\psi^*\d\v)\Xo_\|\\
&={1\over 2}\left(\d\A+\d \b\right)\left(\Xo-\lambda E\right)=
{1\over 2}\left[\alpha+\beta-\left((\alpha+\beta)\cdot\v\right)\v\right],
\end{align*}
and similarly
\begin{align*}
g&=\left(\psi^*\d\w\right)\Xo_\|\\
&=\left[{1\over 2}\left(\A\times \d\b -\b\times \d\A\right)
-x\d \v-\v\d x\right]\left(\Xo-\lambda E\right)\\
&={1\over 2}\left(\A\times\beta-\b\times\alpha\right)-{x\over 2}\left[\alpha+\beta-\left((\alpha+\beta)\cdot\v\right)\v\right]\\
&\qquad-\left({\v\over 2}\cdot(\alpha+\beta)\right)\left[\A\times\b-x\left(\A+\b\right)\right]-\gamma\v, \\
h&=\left(\psi^*\d s\right)\Xo_\|
={1\over 2}\left(\A\cdot\d\A- \b\cdot\d \b\right)\left(\Xo-\lambda E\right)\\
&={1\over 2}\left[\A\cdot\alpha-\b\cdot\beta-\left(\v\cdot(\alpha+\beta)\right)\left(|\A|^2-|\b|^2\right)\right].
\end{align*}
\mn 

Now we apply the above to the vector field $\Xo$ on  $\Im(\tO)$ given by  Equation \eqref{eq:deriv}, with $T=T^t,$ $\bQ=0,$ $\bp=0.$ Then 
$$\alpha=T\A, \ \beta=-T\b,\ \gamma=0.
$$
Using these values in the above expressions for $f,g,h$, we obtain, after some simplification, the stated formulas. 
\end{proof}

\subsection{The $\g_2$-action on $\tQr$ does not descend to $\Qr$}

 It was shown in \cite[\S7]{BM}  that the $\G$-action on $\tQr$ does not 
descend to $\Qr$. Here we show that  the infinitesimal $\g_2$-action does not descend either (a stronger result). Actually, we will determine the elements of $\g_2$ whose action on $\tQr$ descend to $\Qr$. 

\newcommand{\gk}{\mathfrak{k}}

\begin{prop}Consider the $2:1$ cover $\tQr\to \Qr$ and the vector field $\Xr$ on $\tQr$ induced by an element $\rho(T,\bQ, \bp)\in\g_2$, see Equation \eqref{eq:deriv}. Then $\Xr$ descends to a vector field on $\Qr$ if and only if $T^t=-T$, $\bQ=-\bp$. 
\end{prop}

\begin{proof}Recall that $\tQr=S^2\times S^3$, $\Qr=S^2\times\SO_3$ and  $\tQr\to \Qr$ is the quotient by the antipodal map in the second factor. Let us denote this self-map of $\tQr$ by $\sigma$. The vector field $\Xr$ on $\tQr$ thus descends to $\Qr$ if and only if it is $\sigma$-invariant, $\sigma_*\Xr=\Xr.$

The diffeomorphism $\Phi:\tQr\to\tQo$ maps $\Xr$ to a vector field $\Phi_*\Xr$ on $\tQo$. Thus $\Xr$ is $\sigma$-invariant if and only if 
$\Phi_*\Xr$ is $\tau$-invariant, where $\tau=\Phi\circ\sigma\circ \Phi^{-1}.$

\newcommand{\ttau}{\widetilde{\tau}}
Using the definition of $\Phi$ (Equation \eqref{eqn:Phi}), one finds that $\tau$ is given by  a linear involution  of $\Im(\tO)$, 
$$
\ttau: \left(\begin{matrix} x& \A \\ \b& -x
\end{matrix}\right)\mapsto
\left(\begin{matrix} -x& \b \\ \A& x
\end{matrix}\right)
$$
Namely, if $\m\in C\setminus 0$ (a non-zero  null octonion) then $\tau\left([\m]\right)=[\widetilde\tau\m].$

Similarly, the vector field $\Phi_*\Xr$ on $\tQo$ is given by  the linear vector field $\Xo$ on $\Im(\tO)$ of Equation \eqref{eq:deriv} by first restricting $\Xo$ to the null cone $C\subset \Im(\tO)$ (this restriction makes sense since  $\Xo$ is tangent to $C\setminus 0$), then projecting via the $\R^+$-quotient $\pi:C\setminus 0\to \tQo$. 

Summarizing, the $\sigma$-invariance of $\Xr$ amounts to 
\be\label{eq:inva}
\Xo(\ttau\m)\equiv \ttau_*\left(\Xo(\m)\right) \ \mbox{mod } E(\ttau\m), \ \forall \m\in C,
\ee
where $E$ is the Euler vector field on $\Im(\tO)$ (see Equation \eqref{eq:euler}). 

Next recall from \eqref{eq:deriv}  that if $\m=\left(\begin{matrix} x& \A \\ \b& -x
\end{matrix}\right)\in \Im(\tO)$ then 
\begin{multline*}
\Xo(\m)=(T\A -\bp \times \b +2x\bQ )\partial_\A+ (\bQ\times\A -T^t\b + 2x\bp)\partial_\b \\
+ (\bp\cdot\A +\b\cdot\bQ )\partial_x,
\end{multline*}
hence  
\begin{multline*}
\Xo(\ttau\m)=(T\b -\bp \times \A -2x\bQ )\partial_\A+ (\bQ\times\b -T^t\A - 2x\bp)\partial_\b \\
+ (\bp\cdot\b +\A\cdot\bQ )\partial_x 
\end{multline*}
and  
\begin{multline*}
\ttau_\ast\left(\Xo(\m)\right)=(\bQ\times\A -T^t\b + 2x\bp)\partial_\A + (T\A -\bp \times \b +2x\bQ )\partial_\b 
\\ - (\bp\cdot\A +\b\cdot\bQ )\partial_x.
\end{multline*}
Hence Equation \eqref{eq:inva} is equivalent to 
$$\left(\begin{array}{c}
\bQ\times\A -T^t\b + 2x\bp\\
T\A -\bp \times \b +2x\bQ \\
-\bp\cdot\A -\b\cdot\bQ
\end{array}\right)
\equiv 
\left(\begin{array}{c}
T\b -\bp\times\A - 2x\bQ\\
\bQ\times\b -T^t\A - 2x\bp\\
\bp\cdot \b+\A \cdot \bQ
\end{array}\right)
\mbox{ mod } 
\left(\begin{array}{c}
\b\\
\A\\
-x
\end{array}\right)
$$
for all $\A,\b, x$ such that $\b\A=x^2.$ Setting $\b=0$, $x=0$ in the above equation, the  third component reads $(\bQ+\bp)\cdot \A=0$ for all $\A$, hence $\bQ=-\bp.$ The  second component then reads
$(T+T^t)\A=\lambda\A$ for all $\A$ for some $\lambda$ (that may depend on $\A$). Hence $T+T^t$ is a multiple of the identity. But $T$ is traceless, and so is $T+T^t$, hence $T+T^t=0$, as needed. 
\end{proof}

\begin{remark}
The subspace of $\g_2$ indicated in this proposition is isomorphic to $\so_3\oplus\so_3=\so_4$, the Lie algebra of a maximal compact subgroup $\SO_4\simeq K\subset\G$, which has an `obvious' action on $\Qr=S^2\times \SO_3$. In fact, the $\SO_4$-action descends to $\SO_3\times\SO_3$ (a $\Z_2$-quotient of $\SO_4$), $(g_1,g_2): (\v, g)\mapsto (g_1\v, g_1 g g_2^{-1}).$ The embedding $\SO_4\to \G,$ and the associated embedding $\so_3\oplus\so_3\to \g_2$, has appeared explicitly above  during the proof of Theorem 
\ref{thm:diag}; see Section \ref{sec:diag}, item (3), step (c), Equation \eqref{eq:emb}.
This  also appeared in \cite{BM}, \S2.2 and \S5. 
\end{remark}

\section{Coordinate  formulae for $\Dr$}\label{app:coord}

The configuration space $\Qr$ of rolling a sphere of radius $1/\rho$ on a stationary sphere of radius 1 is $S^2\times \SO_3$, where $\v\in S^2$ denotes the contact point of the two spheres and  $g\in\SO_3$ the orientation of the moving sphere with respect to  some fixed initial orientation. 

To write down the rolling distribution on $\Qr$ we use the coordinates $(\phi, \theta, \alpha, \beta, \gamma)$, where $\phi, \theta$ are the spherical coordinates on $S^2$, 
\begin{align*}
\v
&= \begin{pmatrix}
\sin\theta\cos\phi \\
\sin\theta \sin\phi \\
\cos\theta
\end{pmatrix}
\end{align*} 
and $\alpha,\beta, \gamma$
are   Euler angles  on $\SO_3$, 
\begin{align*}
g&= \begin{pmatrix}
\cos \gamma & -\sin \gamma & 0 \\
\sin \gamma & \cos \gamma & 0 \\
0 & 0 & 1
\end{pmatrix} 
\begin{pmatrix}
\cos \beta & 0 & \sin \beta \\
0 & 1 & 0 \\
-\sin \beta & 0 & \cos \beta
\end{pmatrix}
 \begin{pmatrix}
1 & 0 & 0 \\
0 & \cos \alpha & -\sin \alpha \\
0 & \sin \alpha & \cos \alpha
\end{pmatrix}\\
\end{align*}

The angular velocity vector of the moving sphere, about its center,  is 
$$
\boldsymbol{\omega}  =  \begin{pmatrix}
\dot{\alpha} \sin \beta \sin \gamma + \dot{\beta} \cos \gamma \\
\dot{\alpha} \sin \beta \cos \gamma - \dot{\beta} \sin \gamma \\
\dot{\alpha} \cos \beta + \dot{\gamma}
\end{pmatrix}
$$

The rolling distribution $\Dr$ on $\Qr$ is the rank 2 distribution  whose integral curves satisfy $(1+\rho)\dot\v=\boldsymbol{\omega}\times\v, \ \boldsymbol{\omega}\cdot\v=0
$ (see Definition \ref{def:rol}). 
Explicitly, 
\medskip

$\ (1+\rho)
(\dot{\theta} \cos \theta \cos \phi - \dot{\phi} \sin \theta \sin \phi)-
(\dot{\alpha} \sin \beta \sin \gamma + \dot{\beta} \cos \gamma) \sin \theta \cos \phi$\\
\rightline{$+(\dot{\alpha} \sin \beta \cos \gamma - \dot{\beta} \sin \gamma) \sin \theta \sin \phi - \dot{\alpha} \cos \beta \cos \theta=0,$}

\medskip

$\ (1+\rho)
(\dot{\theta} \cos \theta \cos \phi - \dot{\phi} \sin \theta \sin \phi)-
(\dot{\alpha} \sin \beta \sin \gamma + \dot{\beta} \cos \gamma) \sin \theta \cos \phi$\\
\rightline{$+(\dot{\alpha} \sin \beta \cos \gamma - \dot{\beta} \sin \gamma) \sin \theta \sin \phi - \dot{\alpha} \cos \beta \cos \theta=0,$}

\medskip

$\ (1+\rho)(-\dot{\theta} \sin \theta)
+(\dot{\alpha} \sin \beta \sin \gamma + \dot{\beta} \cos \gamma) \cos \theta + \dot{\alpha} \sin \beta \cos \theta=0,$

\medskip

$\ 
\sin\theta\cos\phi(\dot{\alpha} \sin \beta \sin \gamma + \dot{\beta} \cos \gamma)+ 
\sin\theta \sin\phi(\dot{\alpha} \sin \beta \cos \gamma - \dot{\beta} \sin \gamma)$\\
\rightline{$
+\cos\theta(\dot{\alpha} \cos \beta + \dot{\gamma})
=0.$}
\begin{remark}
The  first three equations are linearly dependent, defining a rank 3 distribution on $\Qr$. The relation is $\v\cdot\left[(1+\rho)\dot\v-\boldsymbol{\omega}\times \v\right]=0.$  Thus one can omit say the  first equation, away from the subset of $\Qr$ where $v_1=0$.  
\end{remark}

\end{document}